\documentclass[12pt,reqno]{amsart}
\usepackage{fullpage}

\newtheorem{theorem}{Theorem}[section]
\newtheorem{lemma}[theorem]{Lemma}

\newtheorem{conj}[theorem]{Conjecture}

\newtheorem{prob}[theorem]{Problem}
\usepackage{graphicx}
\usepackage{color}
\usepackage[dvipsnames]{xcolor}
\usepackage{subfigure}
\usepackage{amssymb}
\usepackage{amsmath,mathrsfs}
\usepackage{colonequals}
\usepackage{hyperref}

\theoremstyle{definition}
\newtheorem{definition}[theorem]{Definition}

\newtheorem{remark}[theorem]{Remark}

\renewcommand{\subset}{\subseteq}
\renewcommand{\supset}{\supseteq}
\renewcommand{\epsilon}{\varepsilon}

\newcommand{\abs}[1]{\left|#1\right|}                   
\newcommand{\vnorm}[1]{\left\|#1\right\|}    
\newcommand{\vnormf}[1]{\|#1\|}                         
\newcommand{\vnormt}[1]{\left\|#1\right\|}    

\newcommand{\Z}{\mathbb{Z}}                             

\newcommand{\E}{\mathbb{E}}

\renewcommand{\d}{\mathrm{d}}

\newcommand{\R}{\mathbb{R}}

\newcommand{\embolden}[1]{\textbf {#1}}
\newcommand{\redA}{\Sigma}
\newcommand{\redb}{\partial^{*}}

\newcommand{\sdimn}{n}
\newcommand{\adimn}{n+1}

\begin{document}

\title{Three Candidate Plurality is Stablest for Small Correlations}

\author{Steven Heilman}
\author{Alex Tarter}
\address{Department of Mathematics, University of Southern California, Los Angeles, CA 90089-2532}
\email{stevenmheilman@gmail.com}
\email{atarter@usc.edu}
\date{\today}
\thanks{S. H. is Supported by NSF Grant CCF 1911216}

\begin{abstract}
Using the calculus of variations, we prove the following structure theorem for noise stable partitions: a partition of $n$-dimensional Euclidean space into $m$ disjoint sets of fixed Gaussian volumes that maximize their noise stability must be $(m-1)$-dimensional, if $m-1\leq n$.  In particular, the maximum noise stability of a partition of $m$ sets in $\mathbb{R}^{n}$ of fixed Gaussian volumes is constant for all $n$ satisfying $n\geq m-1$.  From this result, we obtain:
\begin{itemize}
\item[(i)] A proof of the Plurality is Stablest Conjecture for $3$ candidate elections, for all correlation parameters $\rho$ satisfying $0<\rho<\rho_{0}$, where $\rho_{0}>0$ is a fixed constant (that does not depend on the dimension $n$), when each candidate has an equal chance of winning.
\item[(ii)] A variational proof of Borell's Inequality (corresponding to the case $m=2$).
\end{itemize}
The structure theorem answers a question of De-Mossel-Neeman and of Ghazi-Kamath-Raghavendra.  Item (i) is the first proof of any case of the Plurality is Stablest Conjecture of Khot-Kindler-Mossel-O'Donnell (2005) for fixed $\rho$, with the case $\rho\to1^{-}$ being solved recently.  Item (i) is also the first evidence for the optimality of the Frieze-Jerrum semidefinite program for solving MAX-3-CUT, assuming the Unique Games Conjecture.  Without the assumption that each candidate has an equal chance of winning in (i), the Plurality is Stablest Conjecture is known to be false.
\end{abstract}


\maketitle
\setcounter{tocdepth}{1}
\tableofcontents
%
%
%
%
%

\section{Introduction}\label{secintro}

\subsection{An Informal Introduction}

A \textbf{voting method} or \textbf{social choice function} with $m$ candidates and $n$ voters is a function
$$f\colon\{1,\ldots,m\}^{n}\to\{1,\ldots,m\}.$$
From the social choice theory perspective, the input of the function $f$ is a list of votes of $n$ people who are choosing between $m$ candidates.  Each of the $m$ candidates is labelled by the integers $1,\ldots,m$.   If the votes are $x\in\{1,\ldots,m\}^{n}$, then $x_{i}$ denotes the vote of person $i\in\{1,\ldots,n\}$ for candidate $x_{i}\in\{1,\ldots,m\}$.   Given the votes $x\in\{1,\ldots,m\}^{n}$, $f(x)$ is interpreted as the winner of the election.

It is both natural and desirable to find a voting method whose output is most likely to be unchanged after votes are randomly altered.  One could imagine that malicious third parties or miscounting of votes might cause random vote changes, so we desire a voting method $f$ whose output is stable to such changes.  In addition to voting motivations, finding a voting method that is stable to noise has applications to the Unique Games Conjecture \cite{khot07,mossel10,khot15}, to semidefinite programming algorithms such as MAX-CUT \cite{khot07,isaksson11}, to learning theory \cite{feldman12}, etc.  For some surveys on this and related topics, see  \cite{odonnell14b,khot10b,heilman20b}.

The output of a constant function $f$ is never altered by changes to the votes.  Also, if the function $f$ only depends on one of its $n$ inputs, then the output of $f$ is rarely changed by independent random changes to each of the votes.  In these cases, the function $f$ is rather ``undemocratic'' from the perspective of social choice theory.  In the case of a constant function, the outcome of the election does not depend at all on the votes.  In the case of a function that only depends on one of its inputs, the outcome of the election only depends on one voter (so $f$ is called a dictatorship function).

Among ``democratic'' voting methods, it was conjectured in \cite{khot07} and proven in \cite{mossel10} that the majority voting method is the voting method that best preserves the outcome of the election.  Below is an informal statement of the main result of \cite{mossel10}.

\begin{theorem}[\embolden{Majority is Stablest, Informal Version}, {\cite[Theorem 4.4]{mossel10}}]\label{misinf}
Suppose we run an election with a large number $n$ of voters and $m=2$ candidates.  We make the following assumptions about voter behavior and about the election method.
\begin{itemize}
\item Voters cast their votes randomly, independently, with equal probability of voting for either candidate.
\item Each voter has a small influence on the outcome of the election.  (That is, all influences from Definition \ref{infdef} are small for the voting method.)
\item Each candidate has an equal chance of winning the election.
\end{itemize}
Under these assumptions, the majority function is the voting method that best preserves the outcome of the election, when votes have been corrupted independently each with probability less than $1/2$.
\end{theorem}
We say a vote $x_{i}\in\{1,2\}$ is corrupted with probability $0<\delta<1$ when, with probability $\delta$, the vote $x_{i}$ is changed to a uniformly random element of $\{1,2\}$, and with probability $1-\delta$, the vote $x_{i}$ is unchanged.

For a formal statement of Theorem \ref{misinf}, see Theorem \ref{misg} below.

The primary interest of the authors of  \cite{khot07} in Theorem \ref{misinf} was proving optimal hardness of approximation for the MAX-CUT problem.  In the MAX-CUT problem, we are given a finite undirected graph on $n$ vertices, and the objective of the problem is to find a partition of the vertices of the graph into two sets that maximizes the number of edges going between the two sets.  The MAX-CUT problem is MAX-SNP hard, i.e. if $P\neq NP$, there is no polynomial time (in $n$) approximation scheme for this problem.  Nevertheless, there is a randomized polynomial time algorithm  \cite{goemans95} that achieves, in expectation, at least $.87856\ldots$ times the maximum value of the MAX-CUT problem.  This algorithm uses semidefinite programming.  Also, the exact expression for the $.87856\ldots$ constant is
$$\min_{-1\leq\rho\leq 1}\frac{2}{\pi}\frac{\arccos(\rho)}{1-\rho}=.87856\ldots$$

The authors of \cite{khot07} showed that, if the Unique Games Conjecture is true, then Theorem \ref{misinf} implies that the Goemans-Williamson algorithm's .87856$\ldots$ constant of approximation cannot be increased.  Assuming the validity of the Unique Games Conjecture is a fairly standard in complexity theory, though the conjecture remains open.  See \cite{odonnell14b,khot10b} and the references therein for more discussion on this conjecture, and see \cite{khot18} for some recent significant progress.

Theorem \ref{misinf} (i.e. Theorem \ref{misg}) gives a rather definitive statement on the two candidate voting method that is most stable to corruption of votes.  Moreover, the applcation of Theorem \ref{misinf} gives a complete understanding of the optimal algorithm for solving MAX-CUT, assuming the Unique Games Conjecture.  Unfortunately, the proof of Theorem \ref{misinf} says nothing about elections with $m>2$ candidates.  Moreover, Theorem \ref{misinf} fails to prove optimality of the Frieze-Jerrum \cite{frieze95} semidefinite programming algorithm for the MAX-$m$-CUT problem.  In the MAX-$m$-CUT problem, we are given a finite undirected graph on $n$ vertices, and the objective of the problem is to find a partition of the vertices of the graph into $m$ sets that maximizes the number of edges going between the two sets.  So, MAX-CUT is the same as MAX-2-CUT.

In order to prove the optimality of the Frieze-Jerrum \cite{frieze95} semidefinite programming algorithm for the MAX-$m$-CUT problem, one would need an analogue of Theorem \ref{misinf} for $m>2$ voters, where the plurality function replaces the majority function.  For this reason, it was conjectured \cite{khot07,isaksson11} that the plurality function is the voting method that is most stable to independent, random vote corruption.

\begin{conj}[\embolden{Plurality is Stablest, Informal Version}, {\cite{khot07}, \cite[Conjecture 1.9]{isaksson11}}]\label{pisinf}
Suppose we run an election with a large number $n$ of voters and $m\geq3$ candidates.  We make the following assumptions about voter behavior and about the election method.
\begin{itemize}
\item Voters cast their votes randomly, independently, with equal probability of voting for each candidate.
\item Each voter has a small influence on the outcome of the election.  (That is, all influences from Definition \ref{infdef} are small for the voting method.)
\item Each candidate has an equal chance of winning the election.
\end{itemize}
Under these assumptions, the plurality function is the voting method that best preserves the outcome of the election, when votes have been corrupted independently each with probability less than $1/2$.
\end{conj}
We say a vote $x_{i}\in\{1,\ldots,m\}$ is corrupted with probability $0<\delta<1$ when, with probability $\delta$, the vote $x_{i}$ is changed to a uniformly random element of $\{1,\ldots,m\}$, and with probability $1-\delta$, the vote $x_{i}$ is unchanged.

In the case that the probability of vote corruption goes to zero, the first author proved the first known case of Conjecture \ref{pisinf} in \cite{heilman18b}, culminating a series of previous works \cite{colding12a,mcgonagle15,barchiesi16,heilman17,milman18a,milman18b,heilman18}.  Conjecture \ref{pisinf} for all fixed parameters $0<\rho<1$ was entirely open until now.  Unlike the case of the Majority is Stablest (Theorem \ref{misg}), Conjecture \ref{pisinf} \textit{cannot hold} when the candidates have unequal chances of winning the election \cite{heilman14}.  This realization is an obstruction to proving Conjecture \ref{pisinf}.  It suggested that existing proof methods for Theorem \ref{misg} cannot apply to Conjecture \ref{pisinf}.

Nevertheless, we are able to overcome this obstruction in the present work.

\begin{theorem}[\embolden{Main Result, Informal Version}]\label{pis3inf}
There exists $\epsilon>0$ such that Conjecture \ref{pisinf} holds for $m=3$ candidates, for all $n\geq1$, when the probability of a single vote being corrupted is any number in the range $(1/2-\epsilon,1/2)$.
\end{theorem}

Theorem \ref{pis3inf} is the first proven case of the Plurality is Stablest Conjecture.\ref{pisinf}.

\subsection{More Formal Introduction}

Using a generalization of the Central Limit Theorem known as the invariance principle \cite{mossel10,isaksson11}, there is an equivalence between the discrete problem of Conjecture \ref{pisinf} and a continuous problem which is known as the Standard Simplex Conjecture \cite{isaksson11}.  For more details on this equivalence, see Section 7 of \cite{isaksson11}We begin by providing some background for the latter conjecture, stated in Conjecture \ref{conj2} below.

For any $k\geq1$, we define the Gaussian density as
\begin{equation}\label{zero0.0}
\begin{aligned}
\gamma_{k}(x)&\colonequals (2\pi)^{-k/2}e^{-\vnormt{x}^{2}/2},\qquad
\langle x,y\rangle\colonequals\sum_{i=1}^{\adimn}x_{i}y_{i},\qquad
\vnormt{x}^{2}\colonequals\langle x,x\rangle,\\
&\qquad\forall\,x=(x_{1},\ldots,x_{\adimn}),y=(y_{1},\ldots,y_{\adimn})\in\R^{\adimn}.
\end{aligned}
\end{equation}

Let $z_{1},\ldots,z_{m}\in\R^{\adimn}$ be the vertices of a regular simplex in $\R^{\adimn}$ centered at the origin.  For any $1\leq i\leq m$, define
\begin{equation}\label{wdef}
\Omega_{i}\colonequals\{x\in\R^{\adimn}\colon\langle x,z_{i}\rangle=\max_{1\leq j\leq m}\langle x,z_{j}\rangle\}.
\end{equation}
We refer to any sets satisfying \eqref{wdef} as \textbf{cones over a regular simplex}.

Let $f\colon\R^{\adimn}\to[0,1]$ be measurable and let $\rho\in(-1,1)$.  Define the \textbf{Ornstein-Uhlenbeck operator with correlation $\rho$} applied to $f$ by
\begin{equation}\label{oudef}
\begin{aligned}
T_{\rho}f(x)
&\colonequals\int_{\R^{\adimn}}f(x\rho+y\sqrt{1-\rho^{2}})\gamma_{\adimn}(y)\,\d y\\
&=(1-\rho^{2})^{-(\adimn)/2}(2\pi)^{-(\adimn)/2}\int_{\R^{\adimn}}f(y)e^{-\frac{\vnorm{y-\rho x}^{2}}{2(1-\rho^{2})}}\,\d y,
\qquad\forall x\in\R^{\adimn}.
\end{aligned}
\end{equation}
$T_{\rho}$ is a parametrization of the Ornstein-Uhlenbeck operator, which gives a fundamental solution of the (Gaussian) heat equation
\begin{equation}\label{oup}
\frac{d}{d\rho}T_{\rho}f(x)=\frac{1}{\rho}\Big(-\overline{\Delta} T_{\rho}f(x)+\langle x,\overline{\nabla}T_{\rho}f(x)\rangle\Big),\qquad\forall\,x\in\R^{\adimn}.
\end{equation}
Here $\overline{\Delta}\colonequals\sum_{i=1}^{\adimn}\partial^{2}/\partial x_{i}^{2}$ and $\overline{\nabla}$ is the usual gradient on $\R^{\adimn}$.  $T_{\rho}$ is not a semigroup, but it satisfies $T_{\rho_{1}}T_{\rho_{2}}=T_{\rho_{1}\rho_{2}}$ for all $\rho_{1},\rho_{2}\in(0,1)$.  We have chosen this definition since the usual Ornstein-Uhlenbeck operator is only defined for $\rho\in[0,1]$.

\begin{definition}[\embolden{Noise Stability}]\label{noisedef}
Let $\Omega\subset\R^{\adimn}$ be measurable.  Let $\rho\in(-1,1)$.  We define the \textit{noise stability} of the set $\Omega$ with correlation $\rho$ to be
$$\int_{\R^{\adimn}}1_{\Omega}(x)T_{\rho}1_{\Omega}(x)\gamma_{\adimn}(x)\,\d x
\stackrel{\eqref{oudef}}{=}(2\pi)^{-(\adimn)}(1-\rho^{2})^{-(\adimn)/2}\int_{\Omega}\int_{\Omega}e^{\frac{-\|x\|^{2}-\|y\|^{2}+2\rho\langle x,y\rangle}{2(1-\rho^{2})}}\,\d x\d y.$$
Equivalently, if $X=(X_{1},\ldots,X_{\adimn}),Y=(Y_{1},\ldots,Y_{\adimn})\in\R^{\adimn}$ are $(\adimn)$-dimensional jointly Gaussian distributed random vectors with $\E X_{i}Y_{j}=\rho\cdot1_{(i=j)}$ for all $i,j\in\{1,\ldots,\adimn\}$, then
$$\int_{\R^{\adimn}}1_{\Omega}(x)T_{\rho}1_{\Omega}(x)\gamma_{\adimn}(x)\,\d x=\mathbb{P}((X,Y)\in \Omega\times \Omega).$$
\end{definition}

Maximizing the noise stability of a Euclidean partition is the continuous analogue of finding a voting method that is most stable to random corruption of votes, among voting methods where each voter has a small influence on the election's outcome.

\begin{prob}[\embolden{Standard Simplex Problem}, {\cite{isaksson11}}]\label{prob2}
Let $m\geq3$.  Fix $a_{1},\ldots,a_{m}>0$ such that $\sum_{i=1}^{m}a_{i}=1$.  Fix $\rho\in(0,1)$.  Find measurable sets $\Omega_{1},\ldots\Omega_{m}\subset\R^{\adimn}$ with $\cup_{i=1}^{m}\Omega_{i}=\R^{\adimn}$ and $\gamma_{\adimn}(\Omega_{i})=a_{i}$ for all $1\leq i\leq m$ that maximize
$$\sum_{i=1}^{m}\int_{\R^{\adimn}}1_{\Omega_{i}}(x)T_{\rho}1_{\Omega_{i}}(x)\gamma_{\adimn}(x)\,\d x,$$
subject to the above constraints.  (Here $\gamma_{\adimn}(\Omega_{i})\colonequals\int_{\Omega_{i}}\gamma_{\adimn}(x)\,\d x$ $\forall$ $1\leq i\leq m$.)
\end{prob}

We can now state the continuous version of Conjecture \ref{pisinf}.

\begin{conj}[\embolden{Standard Simplex Conjecture} {\cite{isaksson11}}]\label{conj2}
Let $\Omega_{1},\ldots\Omega_{m}\subset\R^{\adimn}$ maximize Problem \ref{prob2}.  Assume that $m-1\leq\adimn$.  Fix $\rho\in(0,1)$.  Let $z_{1},\ldots,z_{m}\in\R^{\adimn}$ be the vertices of a regular simplex in $\R^{\adimn}$ centered at the origin.  Then $\exists$ $w\in\R^{\adimn}$ such that, for all $1\leq i\leq m$,
$$\Omega_{i}=w+\{x\in\R^{\adimn}\colon\langle x,z_{i}\rangle=\max_{1\leq j\leq m}\langle x,z_{j}\rangle\}.$$
\end{conj}
It is known that Conjecture \ref{conj2} is false when $(a_{1},\ldots,a_{m})\neq(1/m,\ldots,1/m)$ \cite{heilman14}.  In the remaining case that $a_{i}=1/m$ for all $1\leq i\leq m$, it is assumed that $w=0$ in Conjecture \ref{conj2}.

For expositional simplicity, we separately address the case $\rho<0$ of Conjecture \ref{conj2} in Section \ref{negsec} below.

\subsection{Plurality is Stablest Conjecture}

As previously mentioned, the Standard Simplex Conjecture \cite{isaksson11} stated in Conjecture \ref{conj2} is essentially equivalent to the Plurality is Stablest Conjecture from Conjecture \ref{pisinf}.   After making several definitions, we state a formal version of Conjecture \ref{pisinf} as Conjecture \ref{prob4} below.

If $g\colon\{1,\ldots,m\}^{\sdimn}\to\R$ and $1\leq i\leq\sdimn$, we denote
$$ \E(g)\colonequals m^{-\sdimn}\sum_{\omega\in\{1,\ldots,m\}^{\sdimn}} g(\omega)$$
$$\E_{i}(g)(\omega_{1},\ldots,\omega_{i-1},\omega_{i+1},\ldots,\omega_{\sdimn})\colonequals m^{-1}\sum_{\omega_{i}\in\{1,\ldots,m\}} g(\omega_{1},\ldots,\omega_{n})$$
$$\qquad\qquad\qquad\qquad\qquad\qquad\qquad\qquad\qquad\forall\,(\omega_{1},\ldots,\omega_{i-1},\omega_{i+1},\ldots,\omega_{\sdimn})\in\{1,\ldots,m\}^{\sdimn}.$$
Define also the $i^{th}$ \textbf{influence} of $g$, i.e. the influence of the $i^{th}$ voter of $g$, as
\begin{equation}\label{infdef}
\mathrm{Inf}_{i}(g)\colonequals \E [(g-\E_{i}g)^{2}].
\end{equation}
Let
\begin{equation}\label{deltadef}
\Delta_{m}\colonequals\{(y_{1},\ldots,y_{m})\in\R^{m}\colon y_{1}+\cdots+y_{m}=1,\,\forall\,1\leq i\leq m,\,y_{i}\geq0\}.
\end{equation}
If $f\colon\{1,\ldots,m\}^{\sdimn}\to\Delta_{m}$, we denote the coordinates of $f$ as $f=(f_{1},\ldots,f_{m})$.  For any $\omega\in\Z^{\sdimn}$, we denote $\vnormt{\omega}_{0}$ as the number of nonzero coordinates of $\omega$.  The \textbf{noise stability} of $g\colon\{1,\ldots,m\}^{\sdimn}\to\R$ with parameter $\rho\in(-1,1)$ is
\begin{flalign*}
S_{\rho} g
&\colonequals m^{-\sdimn}\sum_{\omega\in\{1,\ldots,m\}^{\sdimn}} g(\omega)\E_{\rho} g(\delta)\\
&=m^{-\sdimn}\sum_{\omega\in\{1,\ldots,m\}^{\sdimn}} g(\omega)\sum_{\sigma\in\{1,\ldots,m\}^{\sdimn}}\left(\frac{1-(m-1)\rho}{m}\right)^{\sdimn-\vnormt{\sigma-\omega}_{0}}
\left(\frac{1-\rho}{m}\right)^{\vnormt{\sigma-\omega}_{0}} g(\sigma).
\end{flalign*}
Equivalently, conditional on $\omega$, $\E_{\rho}g(\delta)$ is defined so that for all $1\leq i\leq\sdimn$, $\delta_{i}=\omega_{i}$ with probability $\frac{1-(m-1)\rho}{m}$, and $\delta_{i}$ is equal to any of the other $(m-1)$ elements of $\{1,\ldots,m\}$ each with probability $\frac{1-\rho}{m}$, and so that $\delta_{1},\ldots,\delta_{\sdimn}$ are independent.

The \textbf{noise stability} of $f\colon\{1,\ldots,m\}^{\sdimn}\to\Delta_{m}$ with parameter $\rho\in(-1,1)$ is
$$S_{\rho}f\colonequals\sum_{i=1}^{m}S_{\rho}f_{i}.$$

Let $m\geq2$, $k\geq3$.  For each $j\in\{1,\ldots,m\}$, let $e_{j}=(0,\ldots,0,1,0,\ldots,0)\in\R^{m}$ be the $j^{th}$ unit coordinate vector.  Define the \textbf{plurality} function $\mathrm{PLUR}_{m,\sdimn}\colon\{1,\ldots,m\}^{\sdimn}\to\Delta_{m}$ for $m$ candidates and $\sdimn$ voters such that for all $\omega\in\{1,\ldots,m\}^{\sdimn}$.
$$\mathrm{PLUR}_{m,\sdimn}(\omega)
\colonequals\begin{cases}
e_{j}&,\mbox{if }\abs{\{i\in\{1,\ldots,m\}\colon\omega_{i}=j\}}>\abs{\{i\in\{1,\ldots,m\}\colon\omega_{i}=r\}},\\
&\qquad\qquad\qquad\qquad\forall\,r\in\{1,\ldots,m\}\setminus\{j\}\\
\frac{1}{m}\sum_{i=1}^{m}e_{i}&,\mbox{otherwise}.
\end{cases}
$$

We can now state the more formal version of Conjecture \ref{pisinf}.

\begin{conj}[\embolden{Plurality is Stablest, Discrete Version}]\label{prob4}
For any $m\geq2$, $\rho\in[0,1]$, $\epsilon>0$, there exists $\tau>0$ such that if $f\colon\{1,\ldots,m\}^{\sdimn}\to\Delta_{m}$ satisfies $\mathrm{Inf}_{i}(f_{j})\leq\tau$ for all $1\leq i\leq\sdimn$ and for all $1\leq j\leq m$, and if $\E f=\frac{1}{m}\sum_{i=1}^{m}e_{i}$, then
$$
S_{\rho}f\leq \lim_{\sdimn\to\infty}S_{\rho}\mathrm{PLUR}_{m,\sdimn}+\epsilon.
$$
\end{conj}

The main result of the present paper (stated in Theorem \ref{main4} below) is: $\exists$ $\rho_{0}>0$ such that Conjecture \ref{prob4} is true for $m=3$ for all $0<\rho<\rho_{0}$, for all $n\geq1$.  The only previously known case of Conjecture \ref{prob4} was the following.

\begin{theorem}[\embolden{Majority is Stablest, Formal, Biased Case}, {\cite[Theorem 4.4]{mossel10}}]\label{misg}
Conjecture \ref{prob4} is true when $m=2$.
\end{theorem}

For an even more general version of Theorem \ref{misg}, see \cite[Theorem 4.4]{mossel10}.  In particular, the assumption on $\E f$ can be removed, though we know this cannot be done for $m\geq3$ \cite{heilman14}.

\subsection{Our Contribution}

The main structure Theorem below implies that sets optimizing noise stability in Problem \ref{prob2} are inherently low-dimensional.  Though this statement might seem intuitively true, since many inequalities involving the Gaussian measure have low-dimensional optimizers, this statement has not been proven before.  For example, Theorem \ref{mainthm1} was listed as an open question in \cite{de17,de18} and \cite{ghazi18}.  Indeed, the lack of Theorem \ref{mainthm1} has been one main obstruction to a solution of Conjectures \ref{prob2} and \ref{prob4}.

\begin{theorem}[\embolden{Main Structure Theorem/ Dimension Reduction}]\label{mainthm1}
Fix $\rho\in(0,1)$.  Let $m\geq2$ with $m\leq\sdimn+2$.  Let $\Omega_{1},\ldots\Omega_{m}\subset\R^{\adimn}$ maximize Problem \ref{prob2}.  Then, after rotating the sets $\Omega_{1},\ldots\Omega_{m}$ and applying Lebesgue measure zero changes to these sets, there exist measurable sets $\Omega_{1}',\ldots\Omega_{m}'\subset\R^{m-1}$ such that,
$$\Omega_{i}=\Omega_{i}'\times\R^{\sdimn-m+2},\qquad\forall\, 1\leq i\leq m.$$  
\end{theorem}

In the case $m=2$, Theorem \ref{mainthm1} is (almost) a variational proof of Borell's inequality, since it reduces Problem \ref{prob2} to a one-dimensional problem.

In the case $m=3$, Theorem \ref{mainthm1} says that Conjecture \ref{conj2} for arbitrary $\adimn$ reduces to the case $\adimn=2$, which was solved for small $\rho>0$ in \cite{heilman12}.  That is, Theorem \ref{mainthm1} and the main result of \cite{heilman12} imply:

\begin{theorem}[\embolden{Main; Plurality is Stablest for Three Candidates and Small Correlation}]\label{main4}
There exists $\rho_{0}>0$ such that Conjecture \ref{prob4} is true for $m=3$ and for all $0<\rho<\rho_{0}$.
\end{theorem}
In \cite{heilman12} it is noted that $\rho_{0}=e^{-20\cdot 3^{10^{14}}}$ suffices in Theorem \ref{main4}.

We can also prove a version of Theorem \ref{mainthm1} when $\rho<0$.  See Theorem \ref{mainthm1n} and the discussion in Section \ref{negsec}.  One difficulty in proving Theorem \ref{mainthm1} directly for $\rho<0$ is that it is not a priori obvious that a minimizer of Problem \ref{prob2} exists in that case.

\subsection{Noninteractive Simulation of Correlated Distributions}

As mentioned above, Theorem \ref{mainthm1} answers a question in \cite{de17,de18} and \cite{ghazi18}.  Their interest in Theorem \ref{mainthm1} stems from the following problem.  Let $(X,Y)\in\R^{\sdimn}$ be a random vector.  Let $(X_{1},Y_{1}),(X_{2},Y_{2}),\ldots$ be i.i.d. copies of $(X,Y)$.  Suppose there are two players $A$ and $B$.  Player $A$ has access to $X_{1},X_{2},\ldots$ and player $B$ has access to $Y_{1},Y_{2},\ldots$.  Without communication, what joint distributions can players $A$ and $B$ jointly simulate?  For details on the relation of this problem to Theorem \ref{mainthm1}, see \cite{de17,de18} and \cite{ghazi18}.

\subsection{Outline of the Proof of the Structure Theorem}

In this section we outline the proof of Theorem \ref{mainthm1} in the case $m=2$.  The proof loosely follows that of a corresponding statement \cite{mcgonagle15,barchiesi16} for the Gaussian surface area (which was then adapted to multiple sets in \cite{milman18a,milman18b,heilman18}), with a few key differences.  For didactic purposes, we will postpone a discussion of technical difficulties (such as existence and regularity of a maximizer) to Section \ref{secpre}.

Fix $0<a<1$.  Suppose there exists $\Omega,\Omega^{c}\subset\R^{\adimn}$ are measurable sets maximizing
$$\int_{\R^{\adimn}}1_{\Omega}(x)T_{\rho}1_{\Omega}(x)\gamma_{\adimn}(x)dx,$$
subject to the constraint $\gamma_{\adimn}(\Omega)=a$.  A first variation argument (Lemma \ref{latelemma3} below) implies that $\Sigma\colonequals\partial\Omega$ is a level set of the Ornstein-Uhlenbeck operator applied to $1_{\Omega}$.  That is, there exists $c\in\R$ such that
\begin{equation}\label{zero1}
\Sigma=\{x\in\R^{\adimn}\colon T_{\rho}1_{\Omega}(x)=c\}.
\end{equation}
Since $\Sigma$ is a level set, a vector perpendicular to the level set is also perpendicular to $\Sigma$.  Denoting $N(x)\in\R^{\adimn}$ as the unit length exterior pointing normal vector to $x\in\partial\Omega$, \eqref{zero1} implies that
\begin{equation}\label{zero1.7}
\overline{\nabla}T_{\rho}1_{\Omega}(x)= -N(x)\vnorm{\overline{\nabla}T_{\rho}1_{\Omega}(x)}.
\end{equation}
(It is not obvious that there must be a negative sign here, but it follows from examining the second variation.)  We now observe how the noise stability of $\Omega$ changes as the set is translated infinitesimally.  Fix $v\in\R^{\adimn}$, and consider the variation of $\Omega$ induced by the constant vector field $v$.  That is, let $\Psi\colon\R^{\adimn}\times(-1,1)\to\R^{\adimn}$ such that $\Psi(x,0)=x$ and such that $\frac{\d}{\d s}|_{s=0}\Psi(x,s)=v$ for all $x\in\R^{\adimn},s\in(-1,1)$.  For any $s\in(-1,1)$, let $\Omega^{(s)}=\Psi(\Omega,s)$.    Note that $\Omega^{(0)}=\Omega$.  Denote $f(x)\colonequals \langle v,N(x)\rangle$ for all $x\in\Sigma$.  Then define
$$
S(f)(x)\colonequals (1-\rho^{2})^{-(\adimn)/2}(2\pi)^{-(\adimn)/2}\int_{\Sigma}f(y)e^{-\frac{\vnorm{y-\rho x}^{2}}{2(1-\rho^{2})}}\,\d y,\qquad\forall\,x\in\Sigma.
$$
A second variation argument (Lemma \ref{lemma7p} below) implies that, if $f$ is Gaussian volume-preserving, i.e. $\int_{\Sigma}f(x)\gamma_{\adimn}(x)\,\d x=0$, then
\begin{equation}\label{zero4.5}
\begin{aligned}
&\frac{1}{2}\frac{\d^{2}}{\d s^{2}}\Big|_{s=0}\int_{\R^{\adimn}}1_{\Omega^{(s)}}(x)T_{\rho}1_{\Omega^{(s)}}(x)\gamma_{\adimn}(x)\,\d x\\
&\qquad\qquad\qquad\qquad=\int_{\Sigma}\Big(S(f)(x)-\vnorm{\overline{\nabla}T_{\rho}1_{\Omega}(x)}f(x)\Big)f(x)\gamma_{\adimn}(x)\,\d x.
\end{aligned}
\end{equation}
Somewhat unexpectedly, the function $f(x)=\langle v,N(x)\rangle$ is almost an eigenfunction of the operator $S$ (by Lemma \ref{treig}), in the sense that
\begin{equation}\label{zero5}
S(f)(x)=\frac{1}{\rho}f(x)\vnorm{\overline{\nabla} T_{\rho}1_{\Omega}(x)},\qquad\forall\,x\in\Sigma.
\end{equation}
Equation \eqref{zero5} is the \textit{key fact} use in the proof of the main theorem, Theorem \ref{mainthm1}.  Equation \eqref{zero5} follows from \eqref{zero1.7} and the divergence theorem (see Lemma \ref{treig} for a proof of \eqref{zero5}.)  Plugging \eqref{zero5} into \eqref{zero4.5},
\begin{equation}\label{zero3}
\begin{aligned}
\int_{\Sigma}\langle v,N(x)\rangle\gamma_{\adimn}(x)\,\d x=0\quad\Longrightarrow\quad
&\frac{1}{2}\frac{\d^{2}}{\d s^{2}}\Big|_{s=0}\int_{\R^{\adimn}}1_{\Omega^{(s)}}(x)T_{\rho}1_{\Omega^{(s)}}(x)\gamma_{\adimn}(x)\,\d x\\
&\qquad=\Big(\frac{1}{\rho}-1\Big)\int_{\Sigma}\langle v,N(x)\rangle^{2}\vnorm{\overline{\nabla}T_{\rho}1_{\Omega}(x)}\gamma_{\adimn}(x)\,\d x.
\end{aligned}
\end{equation}
The set
$$V\colonequals\Big\{v\in\R^{\adimn}\colon \int_{\Sigma}\langle v,N(x)\rangle\gamma_{\adimn}(x)\,\d x=0\Big\}$$
has dimension at least $\sdimn$, by the rank-nullity theorem.  Since $\Omega$ maximizes noise stability, the quantity on the right of \eqref{zero3} must be non-positive for all $v\in V$, implying that $f=0$ on $\Sigma$ (except possibly on a set of measure zero on $\Sigma$).  (One can show that $\vnorm{\overline{\nabla}T_{\rho}1_{\Omega}(x)}>0$ for all $x\in\Sigma$.  See Lemma \ref{lemma7r}.)  That is, for all $v\in V$, $\langle v,N(x)\rangle=0$ for all $x\in\Sigma$ (except possibly on a set of measure zero on $\Sigma$).  Since $V$ has dimension at least $\sdimn$, there exists a measurable discrete set $\Omega'\subset\R$ such that $\Omega=\Omega'\times\R^{\sdimn}$ after rotating $\Omega$, concluding the proof of Theorem \ref{mainthm1} in the case $m=2$.

Theorem \ref{mainthm1} follows from the realization that all of the above steps still hold for arbitrary $m$ in Conjecture \ref{prob2}.  In particular, the key lemma \eqref{zero5} still holds.  See Lemmas \ref{treig} and \ref{treig2} below.

\begin{remark}
In the case that we replace the Gaussian noise stability of $\Omega$ with the Euclidean heat content
$$
\int_{\R^{\adimn}}1_{\Omega}(x)P_{t}1_{\Omega}(x)\,\d x,\qquad\forall\,t>0
$$
$$
P_{t}f(x)
\colonequals\int_{\R^{\adimn}}f(x+y\sqrt{t})\gamma_{\adimn}(y)\,\d y,
\qquad\forall x\in\R^{\adimn},\qquad\forall\,f\colon\R^{\adimn}\to[0,1],$$
then the corresponding operator $\overline{S}$ from the second variation of the Euclidean heat content satisfies
$$
\overline{S}(f)(x)\colonequals t^{(\adimn)/2}(2\pi)^{-(\adimn)/2}\int_{\Sigma}f(y)e^{-\frac{\vnorm{y- x}^{2}}{2t}}\,\d y,\qquad\forall\,x\in\Sigma,
$$
and then the analogue of \eqref{zero4.5} for $f(x)\colonequals\langle v,N(x)\rangle$ is
$$\overline{S}(f)(x)=f(x)\vnorm{\overline{\nabla} P_{t}1_{\Omega}(x)},\qquad\forall\,x\in\Sigma,$$
so that the second variation corresponding to $f=\langle v,N\rangle$ is automatically zero.  This fact is expected, since a translation does not change the Euclidean heat content.  However, this example demonstrates that the key property of the above proof is exactly \eqref{zero5}.  More specifically, $f$ is an ``almost eigenfunction'' of $S$ with ``eigenvalue'' $1/\rho$ that is larger than $1$.  It seems plausible that other semigroups could also satisfy an identity such as \eqref{zero5}, since \eqref{zero5} seems related to hypercontractivity.  We leave this open for further research.
\end{remark}%

\section{Existence and Regularity}

\subsection{Preliminaries and Notation}\label{secpre}

We say that $\Sigma\subset\R^{\adimn}$ is an $\sdimn$-dimensional $C^{\infty}$ manifold with boundary if $\Sigma$ can be locally written as the graph of a $C^{\infty}$ function on a relatively open subset of $\{(x_{1},\ldots,x_{\sdimn})\in\R^{\sdimn}\colon x_{\sdimn}\geq0\}$.  For any $(\adimn)$-dimensional $C^{\infty}$ manifold $\Omega\subset\R^{\adimn}$ such that $\partial\Omega$ itself has a boundary, we denote
\begin{equation}\label{c0def}
\begin{aligned}
C_{0}^{\infty}(\Omega;\R^{\adimn})
&\colonequals\{f\colon \Omega\to\R^{\adimn}\colon f\in C^{\infty}(\Omega;\R^{\adimn}),\, f(\partial\partial \Omega)=0,\\
&\qquad\qquad\qquad\exists\,r>0,\,f(\Omega\cap(B(0,r))^{c})=0\}.
\end{aligned}
\end{equation}
We also denote $C_{0}^{\infty}(\Omega)\colonequals C_{0}^{\infty}(\Omega;\R)$.  We let $\mathrm{div}$ denote the divergence of a vector field in $\R^{\adimn}$.  For any $r>0$ and for any $x\in\R^{\adimn}$, we let $B(x,r)\colonequals\{y\in\R^{\adimn}\colon\vnormt{x-y}\leq r\}$ be the closed Euclidean ball of radius $r$ centered at $x\in\R^{\adimn}$.  Here $\partial\partial\Omega$ refers to the $(\sdimn-1)$-dimensional boundary of $\Omega$.

\begin{definition}[\embolden{Reduced Boundary}]\label{rbdef}
A measurable set $\Omega\subset\R^{\adimn}$ has \embolden{locally finite surface area} if, for any $r>0$,
$$\sup\left\{\int_{\Omega}\mathrm{div}(X(x))\,\d x\colon X\in C_{0}^{\infty}(B(0,r),\R^{\adimn}),\, \sup_{x\in\R^{\adimn}}\vnormt{X(x)}\leq1\right\}<\infty.$$
Equivalently, $\Omega$ has locally finite surface area if $\nabla 1_{\Omega}$ is a vector-valued Radon measure such that, for any $x\in\R^{\adimn}$, the total variation
$$
\vnormt{\nabla 1_{\Omega}}(B(x,1))
\colonequals\sup_{\substack{\mathrm{partitions}\\ C_{1},\ldots,C_{m}\,\mathrm{of}\,B(x,1) \\ m\geq1}}\sum_{i=1}^{m}\vnormt{\nabla 1_{\Omega}(C_{i})}
$$
is finite \cite{cicalese12}.  If $\Omega\subset\R^{\adimn}$ has locally finite surface area, we define the \embolden{reduced boundary} $\redb \Omega$ of $\Omega$ to be the set of points $x\in\R^{\adimn}$ such that
$$N(x)\colonequals-\lim_{r\to0^{+}}\frac{\nabla 1_{\Omega}(B(x,r))}{\vnormt{\nabla 1_{\Omega}}(B(x,r))}$$
exists, and it is exactly one element of $S^{\sdimn}\colonequals\{x\in\R^{\adimn}\colon\vnorm{x}=1\}$.
\end{definition}

The reduced boundary $\redb\Omega$ is a subset of the topological boundary $\partial\Omega$.  Also, $\redb\Omega$ and $\partial\Omega$ coincide with the support of $\nabla 1_{\Omega}$, except for a set of $\sdimn$-dimensional Hausdorff measure zero.

Let $\Omega\subset\R^{\adimn}$ be an $(\adimn)$-dimensional $C^{2}$ submanifold with reduced boundary $\Sigma\colonequals\redb \Omega$.  Let $N\colon\redA\to S^{\sdimn}$ be the unit exterior normal to $\redA$.  Let $X\in C_{0}^{\infty}(\R^{\adimn},\R^{\adimn})$.  We write $X$ in its components as $X=(X_{1},\ldots,X_{\adimn})$, so that $\mathrm{div}X=\sum_{i=1}^{\adimn}\frac{\partial}{\partial x_{i}}X_{i}$.  Let $\Psi\colon\R^{\adimn}\times(-1,1)\to\R^{\adimn}$ such that
\begin{equation}\label{nine2.3}
\Psi(x,0)=x,\qquad\qquad\frac{\d}{\d s}\Psi(x,s)=X(\Psi(x,s)),\quad\forall\,x\in\R^{\adimn},\,s\in(-1,1).
\end{equation}
For any $s\in(-1,1)$, let $\Omega^{(s)}\colonequals\Psi(\Omega,s)$.  Note that $\Omega^{(0)}=\Omega$.  Let $\Sigma^{(s)}\colonequals\redb\Omega^{(s)}$, $\forall$ $s\in(-1,1)$.
\begin{definition}
We call $\{\Omega^{(s)}\}_{s\in(-1,1)}$ as defined above a \embolden{variation} of $\Omega\subset\R^{\adimn}$.  We also call $\{\Sigma^{(s)}\}_{s\in(-1,1)}$ a \embolden{variation} of $\Sigma=\redb\Omega$.
\end{definition}

For any $x\in\R^{\adimn}$ and any $s\in(-1,1)$, define
\begin{equation}\label{two9c}
V(x,s)\colonequals\int_{\Omega^{(s)}}G(x,y)\,\d y.
\end{equation}

Below, when appropriate, we let $\,\d x$ denote Lebesgue measure, restricted to a surface $\redA\subset\R^{\adimn}$.

\begin{lemma}[\embolden{Existence of a Maximizer}]\label{existlem}
Let $0<\rho<1$ and let $m\geq2$.  Then there exist measurable sets $\Omega_{1},\ldots,\Omega_{m}$ maximizing Problem \ref{prob2}.
\end{lemma}
\begin{proof}
Define $\Delta_{m}$ as in \eqref{deltadef}.  Let $f\colon\R^{\adimn}\to\Delta_{m}$.  We write $f$ in its components as $f=(f_{1},\ldots,f_{m})$.  The set $D_{0}\colonequals\{f\colon\R^{\adimn}\to\Delta_{m}\}$ is norm closed, bounded and convex, therefore it is weakly compact and convex.  Consider the function
$$C(f)\colonequals\sum_{i=1}^{m}\int_{\R^{\adimn}}f_{i}(x)T_{\rho}f_{i}(x)\gamma_{\adimn}(x)\,\d x.$$
This function is weakly continuous on $D_{0}$, and $D_{0}$ is weakly compact, so there exists $\widetilde{f}\in D_{0}$ such that $C(\widetilde{f})=\max_{f\in D_{0}}C(f)$.  Moreover, $C$ is convex since for any $0<t<1$ and for any $f,g\in D_{0}$,
\begin{flalign*}
&tC(f)+(1-t)C(g)-C(tf+(1-t)g)\\
&\qquad=\sum_{i=1}^{m}\int_{\R^{\adimn}}
\Big(tf_{i}(x)T_{\rho}f_{i}(x)+(1-t)g_{i}(x)T_{\rho}g_{i}(x)\\
&\qquad\qquad\qquad\qquad\qquad-(tf_{i}(x)+(1-t)g_{i}(x))T_{\rho}[tf_{i}(x)+(1-t)g_{i}(x)] \Big)\gamma_{\adimn}(x)\,\d x\\
&\qquad=t(1-t)\sum_{i=1}^{m}\int_{\R^{\adimn}}
\Big((f_{i}(x)-g_{i}(x))T_{\rho}[f_{i}(x)-g_{i}(x)]\Big)\gamma_{\adimn}(x)\,\d x
\geq0.
\end{flalign*}
Here we used that
\begin{equation}\label{zero8.0}
\int_{\R^{\adimn}} h(x)T_{\rho}h(x)\gamma_{\adimn}(x)\,\d x=\int_{\R^{\adimn}} (T_{\sqrt{\rho}}h(x))^{2}\gamma_{\adimn}(x)\,\d x\geq0,
\end{equation}
for all measurable $h\colon\R^{\adimn}\to[-1,1]$.

Since $C$ is convex, its maximum must be achieved at an extreme point of $D_{0}$.  Let $e_{1},\ldots,e_{m}$ denote the standard basis of $\R^{m}$, so that $f$ takes its values in $\{e_{1},\ldots,e_{m}\}$. Then, for any $1\leq i\leq m$, define $\Omega_{i}\colonequals\{x\in\R^{\adimn}\colon f(x)=e_{i}\}$, so that $f_{i}=1_{\Omega_{i}}$ $\forall$ $1\leq i\leq m$.
\end{proof}


\begin{lemma}[\embolden{Regularity of a Maximizer}]\label{reglem}
Let $\Omega_{1},\ldots,\Omega_{m}\subset\R^{\adimn}$ be the measurable sets maximizing Problem \ref{prob2}, guaranteed to exist by Lemma \ref{existlem}.  Then the sets $\Omega_{1},\ldots,\Omega_{m}$ have locally finite surface area.  Moreover, for all $1\leq i\leq m$ and for all $x\in\partial\Omega_{i}$, there exists a neighborhood $U$ of $x$ such that $U\cap \partial\Omega_{i}$ is a finite union of $C^{\infty}$ $\sdimn$-dimensional manifolds.
\end{lemma}
\begin{proof}
This follows from a first variation argument and the strong unique continuation property for the heat equation.  We first claim that there exist constants $(c_{ij})_{1\leq i<j\leq m}$ such that
\begin{equation}\label{zero8}
\Omega_{i}\supset\{x\in\R^{\adimn}\colon T_{\rho}1_{\Omega_{i}}(x)>T_{\rho}1_{\Omega_{j}}(x)+c_{ij},\,\forall\,j\in\{1,\ldots,m\}\setminus\{i\}\},\qquad\forall\,1\leq i\leq m.
\end{equation}

By the Lebesgue density theorem \cite[1.2.1, Proposition 1]{stein70}, we may assume that, for all $i\in\{1,\ldots,k\}$, if $y\in \Omega_{i}$, then we have $\lim_{r\to0}\gamma_{\adimn}(\Omega_{i}\cap B(y,r))/\gamma_{\adimn}(B(y,r))=1$.

We prove \eqref{zero8} by contradiction.  Suppose there exist $c\in\R$, $j,k\in\{1,\ldots,m\}$ with $j\neq k$ and there exists $y\in\Omega_{j}$ and $z\in\Omega_{k}$ such that
$$T_{\rho}(1_{\Omega_{j}}-1_{\Omega_{k}})(y)<c,\qquad T_{\rho}(1_{\Omega_{j}}-1_{\Omega_{k}})(z)>c.$$
By \eqref{oudef}, $T_{\rho}(1_{\Omega_{j}}-1_{\Omega_{k}})(x)$ is a continuous function of $x$.  And by the Lebsgue density theorem, there exist disjoint measurable sets $U_{j},U_{k}$ with positive Lebesgue measure such that $U_{j}\subset\Omega_{j},U_{k}\subset\Omega_{k}$ such that $\gamma_{\adimn}(U_{j})=\gamma_{\adimn}(U_{k})$ and such that
\begin{equation}\label{zero9.0}
T_{\rho}(1_{\Omega_{j}}-1_{\Omega_{k}})(y')<c,\,\,\forall\,y'\in U_{j},\qquad
T_{\rho}(1_{\Omega_{j}}-1_{\Omega_{k}})(y')>c,\,\,\forall\,y'\in U_{k}.
\end{equation}
We define a new partition of $\R^{\adimn}$ such that $\widetilde{\Omega}_{j}\colonequals U_{k}\cup \Omega_{j}\setminus U_{j}$, $\widetilde{\Omega}_{k}\colonequals U_{j}\cup \Omega_{k}\setminus U_{k}$, and $\widetilde{\Omega}_{i}\colonequals\Omega_{i}$ for all $i\in\{1,\ldots,m\}\setminus\{j,k\}$.  Then
\begin{flalign*}
&\sum_{i=1}^{m}\int_{\R^{\adimn}}1_{\widetilde{\Omega}_{i}}(x)T_{\rho}1_{\widetilde{\Omega}_{i}}(x)\gamma_{\adimn}(x)\,\d x
-\sum_{i=1}^{m}\int_{\R^{\adimn}}1_{\Omega_{i}}(x)T_{\rho}1_{\Omega_{i}}(x)\gamma_{\adimn}(x)\,\d x\\
&\qquad=\int_{\R^{\adimn}}1_{\widetilde{\Omega}_{j}}(x)T_{\rho}1_{\widetilde{\Omega}_{j}}(x)\gamma_{\adimn}(x)\,\d x
-\int_{\R^{\adimn}}1_{\Omega_{j}}(x)T_{\rho}1_{\Omega_{j}}(x)\gamma_{\adimn}(x)\,\d x\\
&\qquad\qquad\qquad+\int_{\R^{\adimn}}1_{\widetilde{\Omega}_{k}}(x)T_{\rho}1_{\widetilde{\Omega}_{k}}(x)\gamma_{\adimn}(x)\,\d x
-\int_{\R^{\adimn}}1_{\Omega_{k}}(x)T_{\rho}1_{\Omega_{k}}(x)\gamma_{\adimn}(x)\,\d x\\
&\qquad=\int_{\R^{\adimn}}[1_{\Omega_{j}}-1_{U_{j}}+1_{U_{k}}](x)T_{\rho}[1_{\Omega_{j}}-1_{U_{j}}+1_{U_{k}}]\gamma_{\adimn}(x)\,\d x
\\
&\qquad\qquad\qquad+\int_{\R^{\adimn}}[1_{\Omega_{k}}-1_{U_{k}}+1_{U_{j}}]T_{\rho}[1_{\Omega_{k}}-1_{U_{k}}+1_{U_{j}}]\gamma_{\adimn}(x)\,\d x\\
&\qquad\qquad\qquad-\int_{\R^{\adimn}}1_{\Omega_{j}}(x)T_{\rho}1_{\Omega_{j}}(x)\gamma_{\adimn}(x)\,\d x-\int_{\R^{\adimn}}1_{\Omega_{k}}(x)T_{\rho}1_{\Omega_{k}}(x)\gamma_{\adimn}(x)\,\d x\\
&\qquad=2\int_{\R^{\adimn}}[-1_{U_{j}}+1_{U_{k}}](x)T_{\rho}[1_{\Omega_{j}}-1_{\Omega_{k}}]\gamma_{\adimn}(x)\,\d x\\
&\qquad\qquad\qquad+2\int_{\R^{\adimn}}[1_{U_{j}}-1_{U_{k}}]T_{\rho}[1_{U_{j}}-1_{U_{k}}]\gamma_{\adimn}(x)\,\d x
\stackrel{\eqref{zero9.0}\wedge\eqref{zero8.0}}{>}0.
\end{flalign*}
This contradicts the maximality of $\Omega_{1},\ldots,\Omega_{m}$.  We conclude that \eqref{zero8} holds.

We now fix $1\leq i<j\leq m$ and we upgrade \eqref{zero8} by examining the level sets of
$$T_{\rho}(1_{\Omega_{i}}-1_{\Omega_{j}})(x),\qquad\forall\,x\in\R^{\adimn}.$$
Fix $c\in\R$ and consider the level set
$$\Sigma\colonequals\{x\in\R^{\adimn}\colon T_{\rho}(1_{\Omega_{i}}-1_{\Omega_{j}})(x)=c \}.$$
This level set has Hausdorff dimension at most $\sdimn$ by \cite[Theorem 2.3]{chen98}.

From the Strong Unique Continuation Property for the heat equation \cite{lin90}, $T_{\rho}(1_{\Omega_{i}}-1_{\Omega_{j}})(x)$ does not vanish to infinite order at any $x\in\R^{\adimn}$, so the argument of \cite[Lemma 1.9]{hardt89} (see \cite[Proposition 1.2]{lin94} and also \cite[Theorem 2.1]{chen98}) shows that in a neighborhood of each $x\in\Sigma$, $\Sigma$ can be written as a finite union of $C^{\infty}$ manifolds.  That is, there exists a neighborhood $U$ of $x$ and there exists an integer $k\geq1$ such that
$$U\cap\Sigma=\cup_{i=1}^{k}\{y\in U\colon D^{i}T_{\rho}(1_{\Omega_{i}}-1_{\Omega_{j}})(x)\neq 0,\,\, D^{j}T_{\rho}(1_{\Omega_{i}}-1_{\Omega_{j}})(x)=0,\,\,\forall\,1\leq j\leq i-1\}.$$
Here $D^{i}$ denotes the array of all iterated partial derivatives of order $i\geq1$.  We therefore have
$$\Sigma_{ij}\colonequals(\redb\Omega_{i})\cap(\redb\Omega_{j})\supset\{x\in\R^{\adimn}\colon T_{\rho}(1_{\Omega_{i}}-1_{\Omega_{j}})(x)=c_{ij} \},$$
and the Lemma follows.
\end{proof}

From Lemma \ref{reglem} and Definition \ref{rbdef}, for all $1\leq i<j\leq m$, if $x\in\Sigma_{ij}$, then the unit normal vector $N_{ij}(x)\in\R^{\adimn}$ that points from $\Omega_{i}$ into $\Omega_{j}$ is well-defined on $\Sigma_{ij}$, $\big((\partial\Omega_{i})\cap(\partial\Omega_{j})\big)\setminus\Sigma_{ij}$ has Hausdorff dimension at most $\sdimn-1$, and
\begin{equation}\label{zero11}
N_{ij}(x)=\pm\frac{\overline{\nabla} T_{\rho}(1_{\Omega_{i}}-1_{\Omega_{j}})(x)}{\vnorm{\overline{\nabla} T_{\rho}(1_{\Omega_{i}}-1_{\Omega_{j}})(x)}},\qquad\forall\,x\in\Sigma_{ij}.
\end{equation}
In Lemma \ref{lemma7p} below we will show that the negative sign holds in \eqref{zero11} when $\Omega_{1},\ldots,\Omega_{m}$ maximize Problem \ref{prob2}.


%

\section{First and Second Variation}

In this section, we recall some standard facts for variations of sets with respect to the Gaussian measure.  Here is a summary of notation.

\textbf{Summary of Notation}.
\begin{itemize}
\item $T_{\rho}$ denotes the Ornstein-Uhlenbeck operator with correlation parameter $\rho\in(-1,1)$.
\item $\Omega_{1},\ldots,\Omega_{m}$ denotes a partition of $\R^{\adimn}$ into $m$ disjoint measurable sets.
\item $\redb\Omega$ denotes the reduced boundary of $\Omega\subset\R^{\adimn}$.
\item $\Sigma_{ij}\colonequals(\redb\Omega_{i})\cap(\redb\Omega_{j})$ for all $1\leq i,j\leq m$.
\item $N_{ij}(x)$ is the unit normal vector to $x\in\Sigma_{ij}$ that points from $\Omega_{i}$ into $\Omega_{j}$, so that $N_{ij}=-N_{ji}$.
\end{itemize}
Throughout the paper, unless otherwise stated, we define $G\colon\R^{\adimn}\times\R^{\adimn}\to\R$ to be the following function.  For all $x,y\in\R^{\adimn}$, $\forall$ $\rho\in(-1,1)$, define
\begin{equation}\label{gdef}
\begin{aligned}
G(x,y)&=(1-\rho^{2})^{-(\adimn)/2}(2\pi)^{-(\adimn)}e^{\frac{-\|x\|^{2}-\|y\|^{2}+2\rho\langle x,y\rangle}{2(1-\rho^{2})}}\\
&=(1-\rho^{2})^{-(\adimn)/2}\gamma_{\adimn}(x)\gamma_{\adimn}(y)e^{\frac{-\rho^{2}(\|x\|^{2}+\|y\|^{2})+2\rho\langle x,y\rangle}{2(1-\rho^{2})}}\\
&=(1-\rho^{2})^{-(\adimn)/2}(2\pi)^{-(\adimn)/2}\gamma_{\adimn}(x)e^{\frac{-\vnorm{y-\rho x}^{2}}{2(1-\rho^{2})}}.
\end{aligned}
\end{equation}

We can then rewrite the noise stability from Definition \ref{noisedef} as
$$\int_{\R^{\adimn}}1_{\Omega}(x)T_{\rho}1_{\Omega}(x)\gamma_{\adimn}(x)\,\d x
=\int_{\Omega}\int_{\Omega}G(x,y)\,\d x\d y.$$
Our first and second variation formulas for the noise stability will be written in terms of $G$.


\begin{lemma}[\embolden{The First Variation}\,{\cite{chokski07}}; also {\cite[Lemma 3.1, Equation (7)]{heilman14}}]\label{latelemma3}
Let $X\in C_{0}^{\infty}(\R^{\adimn},\R^{\adimn})$.  Let $\Omega\subset\R^{\adimn}$ be a measurable set such that $\partial\Omega$ is a locally finite union of $C^{\infty}$ manifolds.  Let $\{\Omega^{(s)}\}_{s\in(-1,1)}$ be the corresponding variation of $\Omega$.  Then
\begin{equation}\label{Bone6}
\frac{\d}{\d s}\Big|_{s=0}\int_{\R^{\adimn}} 1_{\Omega^{(s)}}(y)G(x,y)\,\d y
=\int_{\partial \Omega}G(x,y)\langle X(y),N(y)\rangle \,\d y.
\end{equation}
\end{lemma}

The following Lemma is a consequence of \eqref{Bone6} and Lemma \ref{reglem}.

\begin{lemma}[\embolden{The First Variation for Maximizers}]\label{firstvarmaxns}
Suppose $\Omega_{1},\ldots,\Omega_{m}\subset\R^{\adimn}$ maximize Problem \ref{prob2}.  Then for all $1\leq i<j\leq m$, there exists $c_{ij}\in\R$ such that
$$T_{\rho}(1_{\Omega_{i}}-1_{\Omega_{j}})(x)=c_{ij},\qquad\forall\,x\in\Sigma_{ij}.$$
\end{lemma}
\begin{proof}
Fix $1\leq i<j\leq m$ and denote $f_{ij}(x)\colonequals\langle X(x),N_{ij}(x)\rangle$ for all $x\in\Sigma_{ij}$.  From Lemma \ref{latelemma3}, if $X$ is nonzero outside of $\Sigma_{ij}$, we get
\begin{flalign*}
&\frac{1}{2}\frac{\d}{\d s}\Big|_{s=0}\sum_{i=1}^{m}\int_{\R^{\adimn}}1_{\Omega_{i}^{(s)}}(x)T_{\rho}1_{\Omega_{i}^{(s)}}(x)\gamma_{\adimn}(x)\,\d x\\
&\qquad=\int_{\Omega_{i}}G(x,y)\int_{\Sigma_{ij}}\langle X(x),N_{ij}(x)\rangle \,\d x \,\d y
+\int_{\Omega_{j}}G(x,y)\int_{\Sigma_{ij}}\langle X(x),N_{ji}(x)\rangle \,\d x \,\d y\\
&\qquad\stackrel{\eqref{oudef}\wedge\eqref{gdef}}{=}\int_{\Sigma_{ij}}T_{\rho}(1_{\Omega_{i}}-1_{\Omega_{j}})(x)f_{ij}(x)\,\d x.
\end{flalign*}
We used above $N_{ij}=-N_{ji}$.  If $T_{\rho}(1_{\Omega_{i}}-1_{\Omega_{j}})(x)$ is nonconstant, then we can construct $f_{ij}$ supported in $\Sigma_{ij}$ with $\int_{\redb\Omega_{i'}}f_{ij}(x)\gamma_{\adimn}(x)dx=0$ for all $1\leq i'\leq m$ to give a nonzero derivative, contradicting the maximality of $\Omega_{1},\ldots,\Omega_{m}$ (as in Lemma \ref{reglem} and \eqref{zero9.0}).
\end{proof}

\begin{theorem}[\embolden{General Second Variation Formula}, {\cite[Theorem 2.6]{chokski07}}; also {\cite[Theorem 1.10]{heilman15}}]\label{thm4}
Let $X\in C_{0}^{\infty}(\R^{\adimn},\R^{\adimn})$.  Let $\Omega\subset\R^{\adimn}$  be a measurable set such that $\partial\Omega$ is a locally finite union of $C^{\infty}$ manifolds.  Let $\{\Omega^{(s)}\}_{s\in(-1,1)}$ be the corresponding variation of $\Omega$.  Define $V$ as in \eqref{two9c}.  Then
\begin{flalign*}
&\frac{1}{2}\frac{\d^{2}}{\d s^{2}}\Big|_{s=0}\int_{\R^{\adimn}} \int_{\R^{\adimn}} 1_{\Omega^{(s)}}(y)G(x,y) 1_{\Omega^{(s)}}(x)\,\d x\d y\\
&\quad=\int_{\redA}\int_{\redA}G(x,y)\langle X(x),N(x)\rangle\langle X(y),N(y)\rangle \,\d x\d y
+\int_{\redA}\mathrm{div}(V(x,0)X(x))\langle X(x),N(x)\rangle \,\d x.
\end{flalign*}

\end{theorem}

\section{Noise Stability and the Calculus of Variations}\label{secnoise}

We now further refine the first and second variation formulas from the previous section.  The following formula follows by using $G(x,y)\colonequals\gamma_{\adimn}(x)\gamma_{\adimn}(y)$ $\forall$ $x,y\in\R^{\adimn}$ in Lemma \ref{latelemma3} and in Theorem \ref{thm4}.

\begin{lemma}[\embolden{Variations of Gaussian Volume}, {\cite{ledoux01}}]\label{lemma41}
Let $\Omega\subset\R^{\adimn}$  be a measurable set such that $\partial\Omega$ is a locally finite union of $C^{\infty}$ manifolds.  Let $X\in C_{0}^{\infty}(\R^{\adimn},\R^{\adimn})$.  Let $\{\Omega^{(s)}\}_{s\in(-1,1)}$ be the corresponding variation of $\Omega$.  Denote $f(x)\colonequals\langle X(x),N(x)\rangle$ for all $x\in\Sigma\colonequals \redb\Omega $.  Then
$$\frac{\d}{\d s}\Big|_{s=0}\gamma_{\adimn}(\Omega^{(s)})=\int_{\Sigma}f(x)\gamma_{\adimn}(x)\,\d x.$$
$$\frac{\d^{2}}{\d s^{2}}\Big|_{s=0}\gamma_{\adimn}(\Omega^{(s)})=\int_{\Sigma}(\mathrm{div}(X)-\langle X,x\rangle)f(x)\gamma_{\adimn}(x)\,\d x.$$
\end{lemma}

\begin{lemma}[\embolden{Extension Lemma for Existence of Volume-Preserving Variations}, {\cite[Lemma 3.9]{heilman18}}]\label{lemma27}
Let $X'\in C_{0}^{\infty}(\R^{\adimn},\R^{\adimn})$ be a vector field.  Define $f_{ij}\colonequals\langle X',N_{ij}\rangle\in C_{0}^{\infty}(\Sigma_{ij})$ for all $1\leq i<j\leq m$.  If
\begin{equation}\label{eight2}
\forall\,1\leq i\leq m,\quad \sum_{j\in\{1,\ldots,m\}\setminus\{i\}}\int_{\Sigma_{ij}}f_{ij}(x)\gamma_{\sdimn}(x)\,\d x=0,
\end{equation}
then $X'|_{\cup_{1\leq i<j\leq m}\Sigma_{ij}}$ can be extended to a vector field $X\in C_{0}^{\infty}(\R^{\adimn},\R^{\adimn})$ such that the corresponding variations $\{\Omega_{i}^{(s)}\}_{1\leq i\leq m,s\in(-1,1)}$ satisfy
$$\forall\,1\leq i\leq m,\quad\forall\,s\in(-1,1),\quad \gamma_{\adimn}(\Omega_{i}^{(s)})=\gamma_{\adimn}(\Omega_{i}).$$
\end{lemma}

\begin{lemma}\label{gpsd}
Define $G$ as in \eqref{gdef}.  Let $f\colon\Sigma\to\R$ be continous and compactly supported.  Then
$$
\int_{\redA}\int_{\redA}G(x,y)f(x)f(y) \,\d x\d y\geq0.
$$
\end{lemma}
\begin{proof}
If $g\colon\R^{\adimn}\to\R$ is continuous and compactly supported, then it is well known that
$$
\int_{\redA}\int_{\redA}G(x,y)g(x)g(y) \,\d x\d y\geq0,
$$
since e.g. $\frac{G(x,y)}{\gamma_{\adimn(x)}\gamma_{\adimn}(y)}$ is the Mehler kernel, which can be written as an (infinite-dimensional) positive semidefinite matrix.  That is, there exists an orthonormal basis $\{\psi_{i}\}_{i=1}^{\infty}$ of $L_{2}(\gamma_{\adimn})$ (of Hermite polynomials) and there exists a sequence of nonnegative real numbers $\{\lambda_{i}\}_{i=1}^{\infty}$ such that the following series converges absolutely pointwise:
$$\frac{G(x,y)}{\gamma_{\adimn}(x)\gamma_{\adimn}(y)}=\sum_{i=1}^{\infty}\lambda_{i}\psi_{i}(x)\psi_{i}(y),\qquad\forall\,x,y\in\R^{\adimn}.$$
From Mercer's Theorem, this is equivalent to : $\forall$ $p\geq1$, for all $z^{(1)},\ldots,z^{(p)}\in\R^{n}$, for all $\beta_{1},\ldots,\beta_{p}\in\R$,
$$\sum_{i,j=1}^{p}\beta_{i}\beta_{j}G(z^{(i)},z^{(j)})\geq0.$$
In particular, this holds for all $z^{(1)},\ldots,z^{(p)}\in\partial\Omega\subset\R^{\adimn}$.  So, the positive semidefinite property carries over (by restriction) to $\partial\Omega$.
\end{proof}

\subsection{Two Sets}\label{twossub}

For didactic purposes, we first present the second variation of noise stability when $m=2$ in Conjecture \ref{prob2}.

\begin{lemma}[\embolden{Second Variation of Noise Stability}]\label{lemma6}
Let $\Omega\subset\R^{\adimn}$  be a measurable set such that $\partial\Omega$ is a locally finite union of $C^{\infty}$ manifolds.  Let $X\in C_{0}^{\infty}(\R^{\adimn},\R^{\adimn})$.  Let $\{\Omega^{(s)}\}_{s\in(-1,1)}$ be the corresponding variation of $\Omega$.  Denote $f(x)\colonequals\langle X(x),N(x)\rangle$ for all $x\in\Sigma\colonequals \redb\Omega $.  Then
\begin{equation}\label{four30}
\begin{aligned}
&\frac{1}{2}\frac{\d^{2}}{\d s^{2}}\Big|_{s=0}\int_{\R^{\adimn}}\int_{\R^{\adimn}} 1_{\Omega^{(s)}}(y)G(x,y) 1_{\Omega^{(s)}}(x)\,\d x\d y
=\int_{\redA}\int_{\redA}G(x,y)f(x)f(y)\,\d x\d y\\
&\qquad\qquad\qquad+\int_{\redA}\langle\overline{\nabla} T_{\rho}1_{\Omega}(x),X(x)\rangle f(x) \gamma_{\adimn}(x)\,\d x\\
&\qquad\qquad\qquad+\int_{\redA} T_{\rho}1_{\Omega}(x)\Big(\mathrm{div}(X(x))-\langle X(x),x\rangle\Big)f(x)\gamma_{\adimn}(x)\,\d x.
\end{aligned}
\end{equation}
\end{lemma}
\begin{proof}
For all $x\in\R^{\adimn}$, we have $V(x,0)\stackrel{\eqref{two9c}}{=}\int_{\Omega}G(x,y)\,\d y\stackrel{\eqref{oudef}}{=}\gamma_{\adimn}(x)T_{\rho}1_{\Omega}(x)$.  So, from Theorem \ref{thm4},
\begin{flalign*}
&\frac{1}{2}\frac{\d^{2}}{\d s^{2}}\Big|_{s=0}\int_{\R^{\adimn}}\int_{\R^{\adimn}} 1_{\Omega^{(s)}}(y)G(x,y) 1_{\Omega^{(s)}}(x)\,\d x\d y\\
&\qquad\qquad\qquad\qquad=\int_{\redA}\int_{\redA}G(x,y)\langle X(x),N(x)\rangle\langle X(y),N(y)\rangle \,\d x\d y\\
&\qquad\qquad\qquad\qquad\qquad+\int_{\redA}(\sum_{i=1}^{\adimn}T_{\rho}1_{\Omega}(x)\frac{\partial}{\partial x_{i}}X_{i}(x)-x_{i}T_{\rho}1_{\Omega}(x)X_{i}(x)\\
&\qquad\qquad\qquad\qquad\qquad\,\,+\frac{\partial}{\partial x_{i}}T_{\rho}1_{\Omega}(x)X_{i}(x))\langle X(x),N(x)\rangle \gamma_{\adimn}(x)\,\d x.
\end{flalign*}
That is, \eqref{four30} holds.
\end{proof}

\begin{lemma}[\embolden{Volume Preserving Second Variation of Maximizers}]\label{lemma7p}
Suppose $\Omega,\Omega^{c}\subset\R^{\adimn}$ maximize Problem \ref{prob2} for $0<\rho<1$ and $m=2$.  Let $\{\Omega^{(s)}\}_{s\in(-1,1)}$ be the corresponding variation of $\Omega$.  Denote $f(x)\colonequals\langle X(x),N(x)\rangle$ for all $x\in\Sigma\colonequals \redb\Omega$.  If
$$\int_{\Sigma}f(x)\gamma_{\adimn}(x)\,\d x=0,$$
Then there exists an extension of the vector field $X|_{\Sigma}$ such that the corresponding variation of $\{\Omega^{(s)}\}_{s\in(-1,1)}$ satisfies
\begin{equation}\label{four32p}
\begin{aligned}
&\frac{1}{2}\frac{\d^{2}}{\d s^{2}}\Big|_{s=0}\int_{\R^{\adimn}}\int_{\R^{\adimn}} 1_{\Omega^{(s)}}(y)G(x,y) 1_{\Omega^{(s)}}(x)\,\d x\d y\\
&\qquad\qquad=\int_{\redA}\int_{\redA}G(x,y)f(x)f(y) \,\d x\d y
-\int_{\redA}\vnorm{\overline{\nabla} T_{\rho}1_{\Omega}(x)}(f(x))^{2} \gamma_{\adimn}(x)\,\d x.
\end{aligned}
\end{equation}
Moreover,
\begin{equation}\label{nabeq2}
\overline{\nabla}T_{\rho}1_{\Omega}(x)=-N(x)\vnorm{\overline{\nabla}T_{\rho}1_{\Omega}(x)},\qquad\forall\,x\in\Sigma.
\end{equation}
\end{lemma}
\begin{proof}
From Lemma \ref{latelemma3}, $T_{\rho}1_{\Omega}(x)$ is constant for all $x\in\Sigma$.  So, from Lemma \ref{lemma41} and Lemma \ref{lemma27}, the last term in \eqref{four30} vanishes, i.e.
\begin{flalign*}
&\frac{1}{2}\frac{\d^{2}}{\d s^{2}}\Big|_{s=0}\int_{\R^{\adimn}} \int_{\R^{\adimn}}1_{\Omega^{(s)}}(y)G(x,y) 1_{\Omega^{(s)}}(x)\,\d x\d y\\
&\qquad\qquad\qquad=\int_{\redA}\int_{\redA}G(x,y)\langle X(x),N(x)\rangle \langle X(y),N(y)\rangle \,\d x\d y\\
&\qquad\qquad\qquad\qquad+\int_{\redA}\langle\overline{\nabla} T_{\rho}1_{\Omega}(x),X(x)\rangle \langle X(x),N(x)\rangle \gamma_{\adimn}(x)\,\d x.
\end{flalign*}
(Here $\overline{\nabla}$ denotes the gradient in $\R^{\adimn}$.)  Since $T_{\rho}1_{\Omega}(x)$ is constant for all $x\in\partial\Omega$ by Lemma \ref{firstvarmaxns}, $\overline{\nabla} T_{\rho}1_{\Omega}(x)$ is parallel to $N(x)$ for all $x\in\partial\Omega$.  That is,
\begin{equation}\label{nabeq}
\overline{\nabla} T_{\rho}1_{\Omega}(x)=\pm\vnorm{\overline{\nabla} T_{\rho}1_{\Omega}(x)}N(x),\qquad\forall\,x\in\partial\Omega.
\end{equation}
In fact, we must have a negative sign in \eqref{nabeq}, otherwise we could find a vector field $X$ supported near $x\in\partial\Omega$ such that \eqref{nabeq} has a positive sign, and then since $G$ is a positive semidefinite function by Lemma \ref{gpsd}, we would have
\begin{flalign*}
&\frac{1}{2}\frac{\d^{2}}{\d s^{2}}\Big|_{s=0}\int_{\R^{\adimn}}\int_{\R^{\adimn}} 1_{\Omega^{(s)}}(y)G(x,y) 1_{\Omega^{(s)}}(x)\,\d x\d y\\
&\qquad\qquad\qquad\qquad\geq\int_{\redA}\langle\overline{\nabla} T_{\rho}1_{\Omega}(x),X(x)\rangle \langle X(x),N(x)\rangle \gamma_{\adimn}(x)\,\d x>0,
\end{flalign*}
a contradiction.  In summary,
\begin{flalign*}
&\frac{1}{2}\frac{\d^{2}}{\d s^{2}}\Big|_{s=0}\int_{\R^{\adimn}}\int_{\R^{\adimn}} 1_{\Omega^{(s)}}(y)G(x,y) 1_{\Omega^{(s)}}(x)\,\d x\d y\\
&\qquad\qquad\qquad=\int_{\redA}\int_{\redA}G(x,y)\langle X(x),N(x)\rangle \langle X(y),N(y)\rangle \,\d x\d y\\
&\qquad\qquad\qquad\qquad-\int_{\redA}\vnorm{\overline{\nabla} T_{\rho}1_{\Omega}(x)}\langle X(x),N(x)\rangle^{2} \gamma_{\adimn}(x)\,\d x.
\end{flalign*}
\end{proof}

\subsection{More than Two Sets}

We can now generalize Section \ref{twossub} to the case of $m>2$ sets.

\begin{lemma}[\embolden{Second Variation of Noise Stability, Multiple Sets}]\label{lemma6v2}
Let $\Omega_{1},\ldots,\Omega_{m}\subset\R^{\adimn}$ be a partition of $\R^{\adimn}$ into measurable sets such that $\partial\Omega_{i}$ is a locally finite union of $C^{\infty}$ manifolds for all $1\leq i\leq m$.  Let $X\in C_{0}^{\infty}(\R^{\adimn},\R^{\adimn})$.  Let $\{\Omega_{i}^{(s)}\}_{s\in(-1,1)}$ be the corresponding variation of $\Omega_{i}$ for all $1\leq i\leq m$.  Denote $f_{ij}(x)\colonequals\langle X(x),N_{ij}(x)\rangle$ for all $x\in\Sigma_{ij}\colonequals (\redb\Omega_{i})\cap(\redb\Omega_{j})$.  We let $N$ denote the exterior pointing unit normal vector to $\redb\Omega_{i}$ for any $1\leq i\leq m$.  Then
\begin{equation}\label{four30v2}
\begin{aligned}
&\frac{1}{2}\frac{\d^{2}}{\d s^{2}}\Big|_{s=0}\sum_{i=1}^{m}\int_{\R^{\adimn}} \int_{\R^{\adimn}} 1_{\Omega_{i}^{(s)}}(y)G(x,y) 1_{\Omega_{i}^{(s)}}(x)\,\d x\d y\\
&\qquad=\sum_{1\leq i<j\leq m}\int_{\Sigma_{ij}}\Big[\Big(\int_{\redb\Omega_{i}}-\int_{\redb\Omega_{j}}\Big)G(x,y)\langle X(y),N(y)\rangle \,\d y\Big] f_{ij}(x) \,\d x\\
&\qquad\qquad+\int_{\Sigma_{ij}}\langle\overline{\nabla} T_{\rho}(1_{\Omega_{i}}-1_{\Omega_{j}})(x),X(x)\rangle f_{ij}(x) \gamma_{\adimn}(x)\,\d x\\
&\qquad\qquad+\int_{\Sigma_{ij}} T_{\rho}(1_{\Omega_{i}}-1_{\Omega_{j}})(x)\Big(\mathrm{div}(X(x))-\langle X(x),x\rangle\Big)f_{ij}(x)\gamma_{\adimn}(x)\,\d x.
\end{aligned}
\end{equation}
\end{lemma}
\begin{proof}
From Lemma \ref{lemma6},
\begin{flalign*}
&\frac{1}{2}\frac{\d^{2}}{\d s^{2}}\Big|_{s=0}\int_{\R^{\adimn}} \int_{\R^{\adimn}} 1_{\Omega_{i}^{(s)}}(y)G(x,y) 1_{\Omega_{i}^{(s)}}(x)\,\d x\d y\\
&\qquad\qquad=\int_{\redb\Omega_{i}}\int_{\redb\Omega_{i}}G(x,y)\langle X(x),N(x)\rangle \langle X(y),N(y)\rangle \,\d x\d y\\
&\qquad\qquad+\int_{\redb\Omega_{i}}\langle\overline{\nabla} T_{\rho}1_{\Omega_{i}}(x),X(x)\rangle \langle X(x),N(x)\rangle \gamma_{\adimn}(x)\,\d x\\
&\qquad\qquad+\int_{\redb\Omega_{i}} T_{\rho}1_{\Omega_{i}}(x)\Big(\mathrm{div}(X(x))-\langle X(x),x\rangle\Big)\langle X(x),N(x)\rangle\gamma_{\adimn}(x)\,\d x.
\end{flalign*}
Summing over $1\leq i\leq m$ and using $N_{ij}=-N_{ji}$ completes the proof.
\end{proof}

Below, we need the following combinatorial Lemma, the case $m=3$ being treated in \cite[Proposition 3.3]{hutchings02}.
\begin{lemma}[{\cite[Lemma 4.6]{heilman18b}}]\label{lemma25.3}
Let $m\geq3$.  Let
$$D_{1}\colonequals \{(x_{ij})_{1\leq i\neq j\leq m}\in\R^{\binom{m}{2}}\colon \forall\,1\leq i\neq j\leq m,\quad x_{ij}=-x_{ji},\,\sum_{j\in\{1,\ldots,m\}\colon j\neq i}x_{ij}=0\}.$$
\begin{flalign*}
D_{2}\colonequals \{(x_{ij})_{1\leq i\neq j\leq m}\in\R^{\binom{m}{2}}
&\colon  \forall\,1\leq i\neq j\leq m,\quad x_{ij}=-x_{ji},\\
&\forall\,1\leq i<j<k\leq m\quad x_{ij}+x_{jk}+x_{ki}=0\}.
\end{flalign*}
Let $x\in D_{1}$ and let $y\in D_{2}$.  Then $\sum_{1\leq i<j\leq m}x_{ij}y_{ij}=0$.
\end{lemma}
\begin{proof}
$D_{1}$ is defined to be perpendicular to vectors in $D_{2}$, and vice versa.  That is, $D_{1}$ and $D_{2}$ are orthogonal complements of each other, and in terms of vector spaces, $D_{1}\oplus D_{2}=\R^{\binom{m}{2}}$.  Consequently, the inner product of any $x\in D_{1}$ and $y\in D_{2}$ is zero.
\end{proof}

\begin{lemma}[\embolden{Volume Preserving Second Variation of Maximizers, Multiple Sets}]\label{lemma7r}
Let $\Omega_{1},\ldots,\Omega_{m}\subset\R^{\adimn}$ be a partition of $\R^{\adimn}$ into measurable sets such that $\partial\Omega_{i}$ is a locally finite union of $C^{\infty}$ manifolds for all $1\leq i\leq m$.  Let $X\in C_{0}^{\infty}(\R^{\adimn},\R^{\adimn})$.  Let $\{\Omega_{i}^{(s)}\}_{s\in(-1,1)}$ be the corresponding variation of $\Omega_{i}$ for all $1\leq i\leq m$.  Denote $f_{ij}(x)\colonequals\langle X(x),N_{ij}(x)\rangle$ for all $x\in\Sigma_{ij}\colonequals (\redb\Omega_{i})\cap(\redb\Omega_{j}) $.  We let $N$ denote the exterior pointing unit normal vector to $\redb\Omega_{i}$ for any $1\leq i\leq m$.  Then
\begin{equation}\label{four32pv2}
\begin{aligned}
&\frac{1}{2}\frac{\d^{2}}{\d s^{2}}\Big|_{s=0}\sum_{i=1}^{m}\int_{\R^{\adimn}} \int_{\R^{\adimn}} 1_{\Omega_{i}^{(s)}}(y)G(x,y) 1_{\Omega_{i}^{(s)}}(x)\,\d x\d y\\
&\qquad\qquad\qquad=\sum_{1\leq i<j\leq m}\int_{\Sigma_{ij}}\Big[\Big(\int_{\redb\Omega_{i}}-\int_{\redb\Omega_{j}}\Big)G(x,y)\langle X(y),N(y)\rangle \,\d y\Big] f_{ij}(x) \,\d x\\
&\qquad\qquad\qquad\qquad-\int_{\Sigma_{ij}}\vnorm{\overline{\nabla} T_{\rho}(1_{\Omega_{i}}-1_{\Omega_{j}})(x)}(f_{ij}(x))^{2} \gamma_{\adimn}(x)\,\d x.
\end{aligned}
\end{equation}
Also,
\begin{equation}\label{nabeq3}
\overline{\nabla}T_{\rho}(1_{\Omega_{i}}-1_{\Omega_{j}})(x)=-N_{ij}(x)\vnorm{\overline{\nabla}T_{\rho}(1_{\Omega_{i}}-1_{\Omega_{j}})(x)},\qquad\forall\,x\in\Sigma_{ij}.
\end{equation}
Moreover, $\vnorm{\overline{\nabla} T_{\rho}(1_{\Omega_{i}}-1_{\Omega_{j}})(x)}>0$ for all $x\in\Sigma_{ij}$, except on a set of Hausdorff dimension at most $\sdimn-1$.
\end{lemma}
\begin{proof}
From Lemma \ref{firstvarmaxns}, there exist constants $(c_{ij})_{1\leq i<j\leq m}$ such that $T_{\rho}(1_{\Omega_{i}}-1_{\Omega_{j}})(x)=c_{ij}$ for all $1\leq i<j\leq m$, for all $x\in\Sigma_{ij}$.  So, from
Lemma \ref{lemma6v2},
\begin{flalign*}
&\frac{1}{2}\frac{\d^{2}}{\d s^{2}}\Big|_{s=0}\sum_{i=1}^{m}\int_{\R^{\adimn}} \int_{\R^{\adimn}} 1_{\Omega_{i}^{(s)}}(y)G(x,y) 1_{\Omega_{i}^{(s)}}(x)\,\d x\d y\\
&\qquad=\sum_{1\leq i<j\leq m}\int_{\Sigma_{ij}}\Big[\Big(\int_{\redb\Omega_{i}}-\int_{\redb\Omega_{j}}\Big)G(x,y)\langle X(y),N(y)\rangle \,\d y\Big] \langle X(x),N_{ij}(x)\rangle \,\d x\\
&\qquad\qquad\qquad+\int_{\Sigma_{ij}}\langle\overline{\nabla} T_{\rho}(1_{\Omega_{i}}-1_{\Omega_{j}})(x),X(x)\rangle \langle X(x),N_{ij}(x)\rangle \gamma_{\adimn}(x)\,\d x\\
&\qquad\qquad\qquad+c_{ij}\int_{\Sigma_{ij}} \Big(\mathrm{div}(X(x))-\langle X(x),x\rangle\Big)\langle X(x),N_{ij}(x)\rangle\gamma_{\adimn}(x)\,\d x.
\end{flalign*}
The last term then vanishes by Lemma \ref{lemma25.3}.  That is,
\begin{flalign*}
&\frac{1}{2}\frac{\d^{2}}{\d s^{2}}\Big|_{s=0}\sum_{i=1}^{m}\int_{\R^{\adimn}} \int_{\R^{\adimn}} 1_{\Omega_{i}^{(s)}}(y)G(x,y) 1_{\Omega_{i}^{(s)}}(x)\,\d x\d y\\
&\qquad\qquad=\sum_{1\leq i<j\leq m}\int_{\Sigma_{ij}}\Big[\Big(\int_{\redb\Omega_{i}}-\int_{\redb\Omega_{j}}\Big)G(x,y)\langle X(y),N(y)\rangle \,\d y\Big] \langle X(x),N_{ij}(x)\rangle \,\d x\\
&\qquad\qquad\qquad+\int_{\Sigma_{ij}}\langle\overline{\nabla} T_{\rho}(1_{\Omega_{i}}-1_{\Omega_{j}})(x),X(x)\rangle \langle X(x),N_{ij}(x)\rangle \gamma_{\adimn}(x)\,\d x.
\end{flalign*}
Meanwhile, if $1\leq i<j\leq m$ is fixed, it follows from Lemma \ref{firstvarmaxns} that
\begin{equation}\label{nabeq2p}
\overline{\nabla} T_{\rho}(1_{\Omega_{i}}-1_{\Omega_{j}})(x)= \pm N_{ij}(x)\vnorm{\overline{\nabla }T_{\rho}(1_{\Omega_{i}}-1_{\Omega_{j}})(x)},\qquad\forall\,x\in\Sigma_{ij}.
\end{equation}
In fact, we must have a negative sign in \eqref{nabeq2p}, otherwise we could find a vector field $X$ supported near $x\in\Sigma_{ij}$ such that \eqref{nabeq} has a positive sign, and then since $G$ is a positive semidefinite function by Lemma \ref{gpsd}, we would have
\begin{flalign*}
&\frac{1}{2}\frac{\d^{2}}{\d s^{2}}\Big|_{s=0}\sum_{i=1}^{m}\int_{\R^{\adimn}} \int_{\R^{\adimn}} 1_{\Omega_{i}^{(s)}}(y)G(x,y) 1_{\Omega_{i}^{(s)}}(x)\,\d x\d y\\
&\qquad\qquad\geq\int_{\Sigma_{ij}}\langle\overline{\nabla} T_{\rho}(1_{\Omega_{i}}-1_{\Omega_{j}})(x),X(x)\rangle \langle X(x),N(x)\rangle \gamma_{\adimn}(x)\,\d x>0,
\end{flalign*}
a contradiction.  In summary,
\begin{flalign*}
&\frac{1}{2}\frac{\d^{2}}{\d s^{2}}\Big|_{s=0}\sum_{i=1}^{m}\int_{\R^{\adimn}} \int_{\R^{\adimn}} 1_{\Omega_{i}^{(s)}}(y)G(x,y) 1_{\Omega_{i}^{(s)}}(x)\,\d x\d y\\
&\qquad\qquad\qquad=\sum_{1\leq i<j\leq m}\int_{\partial\Sigma_{ij}}\Big[\Big(\int_{\partial\Omega_{i}}-\int_{\partial\Omega_{j}}\Big)G(x,y)\langle X(y),N_{ij}(y)\rangle \,\d y\Big] \langle X(x),N_{ij}(x)\rangle \,\d x\\
&\qquad\qquad\qquad\qquad-\int_{\Sigma_{ij}}\vnorm{\overline{\nabla} T_{\rho}(1_{\Omega_{i}}-1_{\Omega_{j}})(x)}\langle X(x),N_{ij}(x)\rangle^{2} \gamma_{\adimn}(x)\,\d x.
\end{flalign*}
\end{proof}

\section{Almost Eigenfunctions of the Second Variation}

For didactic purposes, we first consider the case $m=2$, and we then later consider the case $m>2$.

\subsection{Two Sets}
Let $\Sigma\colonequals\redb\Omega$.  For any bounded measurable $f\colon\Sigma\to\R$, define the following function (if it exists):
\begin{equation}\label{sdef}
S(f)(x)\colonequals (1-\rho^{2})^{-(\adimn)/2}(2\pi)^{-(\adimn)/2}\int_{\Sigma}f(y)e^{-\frac{\vnorm{y-\rho x}^{2}}{2(1-\rho^{2})}}\,\d y,\qquad\forall\,x\in\Sigma.
\end{equation}

\begin{lemma}[\embolden{Key Lemma, $m=2$, Translations as Almost Eigenfunctions}]\label{treig}
Let $\Omega,\Omega^{c}$ maximize Problem \ref{prob2} for $m=2$.  Let $v\in\R^{\adimn}$.  Then
$$S(\langle v,N\rangle)(x)=\langle v,N(x)\rangle\frac{1}{\rho}\vnorm{\overline{\nabla} T_{\rho}1_{\Omega}(x)},\qquad\forall\,x\in\Sigma.$$
\end{lemma}
\begin{proof}
Since $T_{\rho}1_{\Omega}(x)$ is constant for all $x\in\partial\Omega$ by Lemma \ref{firstvarmaxns}, $\overline{\nabla} T_{\rho}1_{\Omega}(x)$ is parallel to $N(x)$ for all $x\in\partial\Omega$.  That is \eqref{nabeq2} holds, i.e.
\begin{equation}\label{firstve}
\overline{\nabla} T_{\rho}1_{\Omega}(x)=-N(x)\vnorm{\overline{\nabla} T_{\rho}1_{\Omega}(x)},\qquad\forall\,x\in\Sigma.
\end{equation}
From Definition \ref{oudef}, and then using the divergence theorem,
\begin{equation}\label{gre}
\begin{aligned}
\langle v,\overline{\nabla} T_{\rho}1_{\Omega}(x)\rangle
&=(1-\rho^{2})^{-(\adimn)/2}(2\pi)^{-(\adimn)/2}\Big\langle v,\int_{\Omega} \overline{\nabla}_{x}e^{-\frac{\vnorm{y-\rho x}^{2}}{2(1-\rho^{2})}}\,\d y\Big\rangle\\
&=(1-\rho^{2})^{-(\adimn)/2}(2\pi)^{-(\adimn)/2}\frac{\rho}{1-\rho^{2}}\int_{\Omega} \langle v,\,y-\rho x\rangle e^{-\frac{\vnorm{y-\rho x}^{2}}{2(1-\rho^{2})}}\,\d y\\
&=-(1-\rho^{2})^{-(\adimn)/2}(2\pi)^{-(\adimn)/2}\rho\int_{\Omega} \mathrm{div}_{y}\Big(ve^{-\frac{\vnorm{y-\rho x}^{2}}{2(1-\rho^{2})}}\Big)\,\d y\\
&=-(1-\rho^{2})^{-(\adimn)/2}(2\pi)^{-(\adimn)/2}\rho\int_{\Sigma}\langle v,N(y)\rangle e^{-\frac{\vnorm{y-\rho x}^{2}}{2(1-\rho^{2})}}\,\d y\\
&\stackrel{\eqref{sdef}}{=}-\rho\, S(\langle v,N\rangle)(x).
\end{aligned}
\end{equation}
Therefore,
$$
\langle v,N(x)\rangle\vnorm{\overline{\nabla} T_{\rho}1_{\Omega}(x)}
\stackrel{\eqref{firstve}}{=}-\langle v,\overline{\nabla} T_{\rho}1_{\Omega}(x)\rangle\\
\stackrel{\eqref{gre}}{=}\rho\, S(\langle v,N\rangle)(x).
$$
\end{proof}
\begin{remark}\label{drk}
To justify the use of the divergence theorem in \eqref{gre}, let $r>0$ and note that we can differentiate under the integral sign of  $T_{\rho}1_{\Omega\cap B(0,r)}(x)$ to get
\begin{equation}\label{grep}
\begin{aligned}
\overline{\nabla} T_{\rho}1_{\Omega\cap B(0,r)}(x)
&=(1-\rho^{2})^{-(\adimn)/2}(2\pi)^{-(\adimn)/2}\Big\langle v,\int_{\Omega\cap B(0,r)} \overline{\nabla}_{x}e^{-\frac{\vnorm{y-\rho x}^{2}}{2(1-\rho^{2})}}\,\d y\Big\rangle\\
&=(1-\rho^{2})^{-(\adimn)/2}(2\pi)^{-(\adimn)/2}\frac{\rho}{1-\rho^{2}}\int_{\Omega\cap B(0,r)} \langle v,\,y-\rho x\rangle e^{-\frac{\vnorm{y-\rho x}^{2}}{2(1-\rho^{2})}}\,\d y\\
&=-(1-\rho^{2})^{-(\adimn)/2}(2\pi)^{-(\adimn)/2}\rho\int_{\Omega\cap B(0,r)} \mathrm{div}_{y}\Big(ve^{-\frac{\vnorm{y-\rho x}^{2}}{2(1-\rho^{2})}}\Big)\,\d y\\
&=-(1-\rho^{2})^{-(\adimn)/2}(2\pi)^{-(\adimn)/2}\rho\int_{(\Sigma\cap B(0,r))\cup(\Omega\cap\partial B(0,r))}\langle v,N(y)\rangle e^{-\frac{\vnorm{y-\rho x}^{2}}{2(1-\rho^{2})}}\,\d y.
\end{aligned}
\end{equation}
Fix $r'>0$.  Fix $x\in\R^{\adimn}$ with $\vnorm{x}<r'$.  The last integral in \eqref{grep} over $\Omega\cap \partial B(0,r)$ goes to zero as $r\to\infty$ uniformly over all such $\vnorm{x}<r'$.  Also
$\overline{\nabla} T_{\rho}1_{\Omega}(x)$
exists a priori for all $x\in\R^{\adimn}$, while
\begin{flalign*}
&\vnorm{\overline{\nabla} T_{\rho}1_{\Omega}(x)-\overline{\nabla} T_{\rho}1_{\Omega\cap B(0,r)}(x)}
\stackrel{\eqref{oudef}}{=}\frac{\rho}{\sqrt{1-\rho^{2}}}\vnorm{\int_{\R^{\adimn}} y 1_{\Omega\cap B(0,r)^{c}}(x\rho+y\sqrt{1-\rho^{2}})\gamma_{\adimn}(y)\,\d y}\\
&\qquad\qquad\qquad
\leq\frac{\rho}{\sqrt{1-\rho^{2}}}\sup_{w\in\R^{\adimn}\colon\vnorm{w}=1}\int_{\R^{\adimn}} \abs{\langle w,y\rangle} 1_{B(0,r)^{c}}(x\rho+y\sqrt{1-\rho^{2}})\gamma_{\adimn}(y)\,\d y.
\end{flalign*}
And the last integral goes to zero as $r\to\infty$, uniformly over all $\vnorm{x}<r'$.
\end{remark}

\begin{lemma}[\embolden{Second Variation of Translations}]
Let $v\in\R^{\adimn}$.  Let $\Omega,\Omega^{c}$ maximize Problem \ref{prob2} for $m=2$.  Let $\{\Omega^{(s)}\}_{s\in(-1,1)}$ be the variation of $\Omega$ corresponding to the constant vector field $X\colonequals v$.  Assume that
$$\int_{\Sigma}\langle v,N(x)\rangle \gamma_{\adimn}(x)\,\d x=0.$$
Then
$$
\frac{1}{2}\frac{\d^{2}}{\d s^{2}}\Big|_{s=0}\int_{\R^{\adimn}}1_{\Omega^{(s)}}(x)T_{\rho}1_{\Omega^{(s)}}(x)\gamma_{\adimn}(x)\,\d x\\
=\Big(\frac{1}{\rho}-1\Big)\int_{\Sigma}\vnormf{\overline{\nabla}T_{\rho}1_{\Omega}(x)}\langle v,N(x)\rangle^{2}\gamma_{\adimn}(x)\,\d x.
$$
\end{lemma}
\begin{proof}
Let $f(x)\colonequals\langle v,N(x)\rangle$ for all $x\in\Sigma$.  From Lemma \ref{lemma7p},
\begin{flalign*}
&\frac{1}{2}\frac{\d^{2}}{\d s^{2}}\Big|_{s=0}\int_{\R^{\adimn}}1_{\Omega^{(s)}}(x)T_{\rho}1_{\Omega^{(s)}}(x)\gamma_{\adimn}(x)\,\d x\\
&\qquad\qquad\qquad=\int_{\Sigma} \Big(S(f)(x)-\vnormf{\overline{\nabla}T_{\rho}1_{\Omega}(x)} f(x)\Big)f(x)\gamma_{\adimn}(x)\,\d x.
\end{flalign*}
Applying Lemma \ref{treig}, $S(f)(x)=f(x)\frac{1}{\rho}\vnormf{\overline{\nabla}T_{\rho}1_{\Omega}(x)}$ $\forall$ $x\in\Sigma$, proving the Lemma.  Note also that $\int_{\Sigma}\vnormf{\overline{\nabla}T_{\rho}1_{\Omega}(x)}\langle v,N(x)\rangle^{2}\gamma_{\adimn}(x)\,\d x$ is finite priori by the divergence theorem and \eqref{nabeq2}:
\begin{flalign*}
\infty&>\abs{\int_{\Omega}\Big\langle v,-x+\nabla\langle v,\overline{\nabla}T_{\rho}1_{\Omega}(x)\rangle\Big\rangle\gamma_{\adimn}(x)\,\d x}
=\abs{\int_{\Omega}\mathrm{div}\Big(v\langle v,\overline{\nabla}T_{\rho}1_{\Omega}(x)\rangle\gamma_{\adimn}(x)\Big)\,\d x}\\
&=\abs{\int_{\Sigma}\langle v,N(x)\rangle\langle v,\overline{\nabla}T_{\rho}1_{\Omega}(x)\rangle\gamma_{\adimn}(x)\,\d x}
\stackrel{\eqref{nabeq2}}{=}\abs{\int_{\Sigma}\vnormf{\overline{\nabla}T_{\rho}1_{\Omega}(x)}\langle v,N(x)\rangle^{2}\gamma_{\adimn}(x)\,\d x}.
\end{flalign*}
\end{proof}

\subsection{More than Two Sets}

Let $v\in\R^{\adimn}$ and denote $f_{ij}\colonequals\langle v,N_{ij}\rangle$ for all $1\leq i,j\leq m$.  For simplicity of notation in the formulas below, if $1\leq i\leq m$ and if a vector $N(x)$ appears inside an integral over $\partial\Omega_{i}$, then $N(x)$ denotes the unit exterior pointing normal vector to $\Omega_{i}$ at $x\in\redb\Omega_{i}$.  Similarly, for simplicity of notation, we denote $\langle v,N\rangle$ as the collection of functions $(\langle v,N_{ij}\rangle)_{1\leq i<j\leq m}$.  For any $1\leq i<j\leq m$, define
\begin{equation}\label{sdef2}
S_{ij}(\langle v,N\rangle)(x)\colonequals (1-\rho^{2})^{-(\adimn)/2}(2\pi)^{-(\adimn)/2}\Big(\int_{\partial\Omega_{i}}-\int_{\partial\Omega_{j}}\Big)\langle v,N(y)\rangle e^{-\frac{\vnorm{y-\rho x}^{2}}{2(1-\rho^{2})}}\,\d y,\,\forall\,x\in\Sigma_{ij}.
\end{equation}
\begin{lemma}[\embolden{Key Lemma, $m\geq 2$, Translations as Almost Eigenfunctions}]\label{treig2}
Let $\Omega_{1},\ldots,\Omega_{m}$ maximize problem \ref{prob2}.  Fix $1\leq i<j\leq m$.  Let $v\in\R^{\adimn}$.  Then
$$S_{ij}(\langle v,N\rangle)(x)=\langle v,N_{ij}(x)\rangle\frac{1}{\rho}\vnorm{\overline{\nabla} T_{\rho}(1_{\Omega_{i}}-1_{\Omega_{j}})(x)},\qquad\forall\,x\in\Sigma_{ij}.$$
\end{lemma}
\begin{proof}
From Lemma \ref{lemma7r}, i.e. \eqref{nabeq3},
\begin{equation}\label{firstve2}
\overline{\nabla} T_{\rho}(1_{\Omega_{i}}-1_{\Omega_{j}})(x)=-N_{ij}(x)\vnorm{\overline{\nabla} T_{\rho}(1_{\Omega_{i}}-1_{\Omega_{j}})(x)},\qquad\forall\,x\in\Sigma_{ij}.
\end{equation}
From Definition \ref{oudef}, and then using the divergence theorem,
\begin{equation}\label{gre2}
\begin{aligned}
\langle v,\overline{\nabla} T_{\rho}1_{\Omega_{i}}(x)\rangle
&=(1-\rho^{2})^{-(\adimn)/2}(2\pi)^{-(\adimn)/2}\Big\langle v,\int_{\Omega_{i}} \overline{\nabla}e^{-\frac{\vnorm{y-\rho x}^{2}}{2(1-\rho^{2})}}\,\d y\Big\rangle\\
&=(1-\rho^{2})^{-(\adimn)/2}(2\pi)^{-(\adimn)/2}\frac{\rho}{1-\rho^{2}}\int_{\Omega_{i}} \langle v,\,y-\rho x\rangle e^{-\frac{\vnorm{y-\rho x}^{2}}{2(1-\rho^{2})}}\,\d y\\
&=-(1-\rho^{2})^{-(\adimn)/2}(2\pi)^{-(\adimn)/2})\rho\int_{\Omega_{i}} \mathrm{div}\Big(ve^{-\frac{\vnorm{y-\rho x}^{2}}{2(1-\rho^{2})}}\Big)\,\d y\\
&=-(1-\rho^{2})^{-(\adimn)/2}(2\pi)^{-(\adimn)/2}\rho\int_{\redb\Omega_{i}}\langle v,N(y)\rangle e^{-\frac{\vnorm{y-\rho x}^{2}}{2(1-\rho^{2})}}\,\d y.
\end{aligned}
\end{equation}
The use of the divergence theorem is justified in Remark \ref{drk}.  Therefore,
\begin{flalign*}
&\langle v,N_{ij}(x)\rangle\vnorm{\overline{\nabla} T_{\rho}(1_{\Omega_{i}}-1_{\Omega_{j}})(x)}
\stackrel{\eqref{firstve2}}{=}-\langle v,\overline{\nabla} T_{\rho}(1_{\Omega_{i}}-1_{\Omega_{j}})(x)\rangle\\
&\qquad\stackrel{\eqref{gre2}}{=}(1-\rho^{2})^{-(\adimn)/2}(2\pi)^{-(\adimn)/2}\rho\Big(\int_{\redb\Omega_{i}}-\int_{\redb\Omega_{j}}\Big)\langle v,N(y)\rangle e^{-\frac{\vnorm{y-\rho x}^{2}}{2(1-\rho^{2})}}\,\d y\\
&\qquad\stackrel{\eqref{sdef2}}{=}\rho\, S_{ij}(\langle v,N\rangle)(x).
\end{flalign*}
\end{proof}

\begin{lemma}[\embolden{Second Variation of Translations, Multiple Sets}]\label{keylem}
Let $v\in\R^{\adimn}$.  Let $\Omega_{1},\ldots,\Omega_{m}$ maximize problem \ref{prob2}.  For each $1\leq i\leq m$, let $\{\Omega_{i}^{(s)}\}_{s\in(-1,1)}$ be the variation of $\Omega_{i}$ corresponding to the constant vector field $X\colonequals v$.  Assume that
$$\int_{\partial\Omega_{i}}\langle v,N(x)\rangle \gamma_{\adimn}(x)\,\d x=0,\qquad\forall\,1\leq i\leq m.$$
Then
\begin{flalign*}
&\frac{1}{2}\frac{\d^{2}}{\d s^{2}}\Big|_{s=0}\sum_{i=1}^{m}\int_{\R^{\adimn}}1_{\Omega_{i}^{(s)}}(x)T_{\rho}1_{\Omega_{i}^{(s)}}(x)\gamma_{\adimn}(x)\,\d x\\
&\qquad\qquad\qquad=\Big(\frac{1}{\rho}-1\Big)\sum_{1\leq i<j\leq m}\int_{\Sigma_{ij}}\vnormf{\overline{\nabla}T_{\rho}(1_{\Omega_{i}}-1_{\Omega_{j}})(x)}\langle v,N_{ij}(x)\rangle^{2}\gamma_{\adimn}(x)\,\d x.
\end{flalign*}
\end{lemma}
\begin{proof}
For any $1\leq i<j\leq m$, let $f_{ij}(x)\colonequals\langle v,N_{ij}(x)\rangle$ for all $x\in\Sigma$.  From Lemma \ref{lemma7p},
\begin{flalign*}
&\frac{1}{2}\frac{\d^{2}}{\d s^{2}}\Big|_{s=0}\sum_{i=1}^{m}\int_{\R^{\adimn}}1_{\Omega^{(s)}}(x)T_{\rho}1_{\Omega^{(s)}}(x)\gamma_{\adimn}(x)\,\d x\\
&\qquad=\sum_{1\leq i<j\leq m}\int_{\Sigma_{ij}} \Big(S_{ij}(\langle v,N\rangle)(x)-\vnormf{\overline{\nabla}T_{\rho}(1_{\Omega_{i}}-1_{\Omega_{j}})(x)} f_{ij}(x)\Big)f_{ij}(x)\gamma_{\adimn}(x)\,\d x.
\end{flalign*}
Applying Lemma \ref{treig2}, $S_{ij}(\langle v,N\rangle)(x)=f_{ij}(x)\frac{1}{\rho}\vnormf{\overline{\nabla}T_{\rho}(1_{\Omega_{i}}-1_{\Omega_{j}})(x)}$, proving the Lemma.  Note also that $\sum_{1\leq i<j\leq m}\int_{\Sigma_{ij}}\vnormf{\overline{\nabla}T_{\rho}(1_{\Omega_{i}}-1_{\Omega_{j}})(x)}\langle v,N_{ij}(x)\rangle^{2}\gamma_{\adimn}(x)\,\d x$ is finite priori by the divergence theorem since
\begin{flalign*}
\infty&>\abs{\int_{\Omega_{i}}\Big\langle v,-x+\nabla\langle v,\overline{\nabla}T_{\rho}1_{\Omega_{i}}(x)\rangle\Big\rangle\gamma_{\adimn}(x)\,\d x}
=\abs{\int_{\Omega_{i}}\mathrm{div}\Big(v\langle v,\overline{\nabla}T_{\rho}1_{\Omega_{i}}(x)\rangle\gamma_{\adimn}(x)\Big)\,\d x}\\
&=\abs{\int_{\Omega_{i}}\mathrm{div}\Big(v\langle v,\overline{\nabla}T_{\rho}1_{\Omega_{i}}(x)\rangle\gamma_{\adimn}(x)\Big)\,\d x}
=\abs{\int_{\redb\Omega_{i}}\langle v,\overline{\nabla}T_{\rho}(1_{\Omega_{i}})(x)\rangle\langle v,N(x)\rangle\gamma_{\adimn}(x)\,\d x}.
\end{flalign*}
Summing over $1\leq i\leq m$ then gives
\begin{flalign*}
\infty>&\abs{\sum_{1\leq i<j\leq m}\int_{\Sigma_{ij}}\langle v,\overline{\nabla}T_{\rho}(1_{\Omega_{i}}-1_{\Omega_{j}})(x)\rangle\langle v,N_{ij}(x)\rangle\gamma_{\adimn}(x)\,\d x.}\\
&\stackrel{\eqref{nabeq3}}{=}\abs{\sum_{1\leq i<j\leq m}\int_{\Sigma_{ij}}\vnormf{\overline{\nabla}T_{\rho}(1_{\Omega_{i}}-1_{\Omega_{j}})(x)}\langle v,N_{ij}(x)\rangle^{2}\gamma_{\adimn}(x)\,\d x.}.
\end{flalign*}
\end{proof}

\section{Proof of the Main Structure Theorem}

\begin{proof}[Proof of Theorem \ref{mainthm1}]
Let $m\geq2$.  Let $0<\rho<1$.    Fix $a_{1},\ldots,a_{m}>0$ such that $\sum_{i=1}^{m}a_{i}=1$.  Let $\Omega_{1},\ldots\Omega_{m}\subset\R^{\adimn}$ be measurable sets that partition $\R^{\adimn}$ such that $\gamma_{\adimn}(\Omega_{i})=a_{i}$ for all $1\leq i\leq m$ that maximize Problem \ref{prob2}.  These sets exist by Lemma \ref{existlem} and from Lemma \ref{reglem} their boundaries are locally finite unions of $C^{\infty}$ $\sdimn$-dimensional manifolds.  Define $\Sigma_{ij}\colonequals(\redb\Omega_{i})\cap(\redb\Omega_{j})$ for all $1\leq i<j\leq m$.

By Lemma \ref{firstvarmaxns}, for all $1\leq i<j\leq m$, there exists $c_{ij}\in\R$ such that
$$T_{\rho}(1_{\Omega_{i}}-1_{\Omega_{j}})(x)=c_{ij},\qquad\forall\,x\in\Sigma_{ij}.$$
By this condition, the regularity Lemma \ref{reglem}, and the last part of Lemma \ref{lemma7r},
$$\overline{\nabla}T_{\rho}(1_{\Omega_{i}}-1_{\Omega_{j}})(x)= -N_{ij}(x)\vnorm{\overline{\nabla}T_{\rho}(1_{\Omega_{i}}-1_{\Omega_{j}})(x)},\qquad\forall\,x\in\Sigma_{ij}.$$
Moreover, by the last part of Lemma \ref{lemma7r}, except for a set $\sigma_{ij}$ of Hausdorff dimension at most $\sdimn-1$, we have
\begin{equation}\label{nine1}
\vnorm{\overline{\nabla}T_{\rho}(1_{\Omega_{i}}-1_{\Omega_{j}})(x)}>0,\qquad\forall\,x\in\Sigma_{ij}\setminus\sigma_{ij}.
\end{equation}

Fix $v\in\R^{\adimn}$, and consider the variation of $\Omega_{1},\ldots,\Omega_{m}$ induced by the constant vector field $X\colonequals v$.  For all $1\leq i<j\leq m$, define $S_{ij}$ as in \eqref{sdef2}.  Define
$$
V\colonequals\Big\{v\in\R^{\adimn}\colon \sum_{j\in\{1,\ldots,m\}\setminus\{i\}}\int_{\Sigma_{ij}}\langle v,N_{ij}(x)\rangle \gamma_{\adimn}(x)\,\d x=0,\qquad\forall\,1\leq i\leq m\Big\}.
$$
From Lemma \ref{keylem},
\begin{flalign*}
v\in V\,\Longrightarrow&\,\,\,\frac{1}{2}\frac{\d^{2}}{\d s^{2}}\Big|_{s=0}\sum_{i=1}^{m}\int_{\R^{\adimn}}1_{\Omega_{i}^{(s)}}(x)T_{\rho}1_{\Omega_{i}^{(s)}}(x)\gamma_{\adimn}(x)\,\d x\\
&\qquad\qquad=\Big(\frac{1}{\rho}-1\Big)\sum_{1\leq i<j\leq m}\int_{\Sigma_{ij}}\vnormf{\overline{\nabla}T_{\rho}(1_{\Omega_{i}}-1_{\Omega_{j}})(x)}\langle v,N_{ij}(x)\rangle^{2}\gamma_{\adimn}(x)\,\d x.
\end{flalign*}
Since $0<\rho<1$, \eqref{nine1} implies
\begin{equation}\label{nine2}
v\in V\,\Longrightarrow\,\langle v,N_{ij}(x)\rangle=0,\qquad\forall\,x\in\Sigma_{ij},\,\forall\,1\leq i<j\leq m.
\end{equation}
The set $V$ has dimension at least $\sdimn+2-m$, by the rank-nullity theorem, since $V$ is the null space of the linear operator $M\colon \R^{\adimn}\to\R^{m}$ defined by
$$
(M(v))_{i}\colonequals \sum_{j\in\{1,\ldots,m\}\setminus\{i\}}\int_{\Sigma_{ij}}\langle v,N_{ij}(x)\rangle \gamma_{\adimn}(x)\,\d x,\qquad\forall\,1\leq i\leq m
$$ 
and $M$ has rank at most $m-1$ (since $\sum_{i=1}^{m}(M(v))_{i}=0$ for all $v\in\R^{\adimn}$).  So, by \eqref{nine2}, after rotating $\Omega_{1},\ldots,\Omega_{m}$, we conclude that there exist measurable $\Omega_{1}',\ldots,\Omega_{m}'\subset\R^{m-1}$ such that
$$\Omega_{i}=\Omega_{i}'\times\R^{\sdimn+2-m},\qquad\forall\,1\leq i\leq m.$$
\end{proof}

\section{The Case of Negative Correlation}\label{negsec}

In this section, we consider the case that $\rho<0$ in Problem \ref{prob2}.  When $\rho<0$ and $h\colon\R^{\adimn}\to[-1,1]$ is measurable, then quantity
$$\int_{\R^{\adimn}}h(x)T_{\rho}h(x)\gamma_{\adimn}(x)\,\d x$$
could be negative, so a few parts of the above argument do not work, namely the existence Lemma \ref{existlem}.  We therefore replace the noise stability with a more general bilinear expression, guaranteeing existence of the corresponding problem.  The remaining parts of the argument are essentially identical, mutatis mutandis.  We indicate below where the arguments differ in the bilinear case.

When $\rho<0$, we look for a minimum of noise stability, rather than a maximum.  Correspondingly, we expect that the plurality function minimizes noise stability when $\rho<0$.  If $\rho<0$, then \eqref{oudef} implies that
$$\int_{\R^{\adimn}}h(x)T_{\rho}h(x)\gamma_{\adimn}(x)\,\d x
=\int_{\R^{\adimn}}h(x)T_{(-\rho)}h(-x)\gamma_{\adimn}(x)\,\d x.$$
So, in order to understand the minimum of noise stability for negative correlations, it suffices to consider the following bilinear version of the standard simplex problem with positive correlation.

\begin{prob}[\embolden{Standard Simplex Problem, Bilinear Version, Positive Correlation}, {\cite{isaksson11}}]\label{prob2n}
Let $m\geq3$.  Fix $a_{1},\ldots,a_{m}>0$ such that $\sum_{i=1}^{m}a_{i}=1$.  Fix $0<\rho<1$.  Find measurable sets $\Omega_{1},\ldots\Omega_{m},\Omega_{1}',\ldots\Omega_{m}'\subset\R^{\adimn}$ with $\cup_{i=1}^{m}\Omega_{i}=\cup_{i=1}^{m}\Omega_{i}'=\R^{\adimn}$ and $\gamma_{\adimn}(\Omega_{i})=\gamma_{\adimn}(\Omega_{i}')=a_{i}$ for all $1\leq i\leq m$ that minimize
$$\sum_{i=1}^{m}\int_{\R^{\adimn}}1_{\Omega_{i}}(x)T_{\rho}1_{\Omega_{i}'}(x)\gamma_{\adimn}(x)\,\d x,$$
subject to the above constraints.
\end{prob}

\begin{conj}[\embolden{Standard Simplex Conjecture, Bilinear Version, Positive Correlation} {\cite{isaksson11}}]\label{conj2n}
Let $\Omega_{1},\ldots\Omega_{m},\Omega_{1}',\ldots\Omega_{m}'\subset\R^{\adimn}$ minimize Problem \ref{prob2}.  Assume that $m-1\leq\adimn$.  Fix $0<\rho<1$.  Let $z_{1},\ldots,z_{m}\in\R^{\adimn}$ be the vertices of a regular simplex in $\R^{\adimn}$ centered at the origin.  Then $\exists$ $w\in\R^{\adimn}$ such that, for all $1\leq i\leq m$,
$$\Omega_{i}=-\Omega_{i}'=w+\{x\in\R^{\adimn}\colon\langle x,z_{i}\rangle=\max_{1\leq j\leq m}\langle x,z_{j}\rangle\}.$$
\end{conj}
In the case that $a_{i}=1/m$ for all $1\leq i\leq m$, it is assumed that $w=0$ in Conjecture \ref{conj2n}.

Since we consider a bilinear version of noise stability in Problem \ref{prob2n}, existence of an optimizer is easier than in Problem \ref{prob2}.

\begin{lemma}[\embolden{Existence of a Minimizer}]\label{existlemn}
Let $0<\rho<1$ and let $m\geq2$.  Then there exist measurable sets $\Omega_{1},\ldots\Omega_{m},\Omega_{1}',\ldots\Omega_{m}'$ that minimize Problem \ref{prob2n}.
\end{lemma}
\begin{proof}
Define $\Delta_{m}$ as in \eqref{deltadef}.  Let $f,g\colon\R^{\adimn}\to\Delta_{m}$.  The set $D_{0}\colonequals\{f\colon\R^{\adimn}\to\Delta_{m}\}$ is norm closed, bounded, and convex, therefore it is weakly compact and convex.  Consider the function
$$C(f,g)\colonequals\sum_{i=1}^{m}\int_{\R^{\adimn}}f_{i}(x)T_{\rho}g_{i}(x)\gamma_{\adimn}(x)\,\d x.$$
This function is weakly continuous on $D_{0}\times D_{0}$, and $D_{0}\times D_{0}$ is weakly compact, so there exists $\widetilde{f},\widetilde{g}\in D_{0}$ such that $C(\widetilde{f},\widetilde{g})=\min_{f,g\in D}C(f,g)$. Since $C$ is bilinear and $D_{0}$ is convex, the minimum of $C$ must be achieved at an extreme point of $D_{0}\times D_{0}$.  Let $e_{1},\ldots,e_{m}$ denote the standard basis of $\R^{m}$, so that $f,g$ takes their values in $\{e_{1},\ldots,e_{m}\}$. Then, for any $1\leq i\leq m$, define $\Omega_{i}\colonequals\{x\in\R^{\adimn}\colon f(x)=e_{i}\}$ and $\Omega_{i}'\colonequals\{x\in\R^{\adimn}\colon g(x)=e_{i}\}$.  Note that $f_{i}=1_{\Omega_{i}}$ and $g_{i}=1_{\Omega_{i}'}$ for all $1\leq i\leq m$.
\end{proof}

\begin{lemma}[\embolden{Regularity of a Minimizer}]\label{reglemn}
Let $\Omega_{1},\ldots,\Omega_{m},\Omega_{1}',\ldots,\Omega_{m}'\subset\R^{\adimn}$ be the measurable sets minimizing Problem \ref{prob2}, guaranteed to exist by Lemma \ref{existlemn}.  Then the sets $\Omega_{1},\ldots,\Omega_{m},\Omega_{1}',\ldots,\Omega_{m}'$ have locally finite surface area.  Moreover, for all $1\leq i\leq m$ and for all $x\in\partial\Omega_{i}$, there exists a neighborhood $U$ of $x$ such that $U\cap \partial\Omega_{i}$ is a finite union of $C^{\infty}$ $\sdimn$-dimensional manifolds.  The same holds for $\Omega_{1}',\ldots,\Omega_{m}'$.
\end{lemma}

We denote $\Sigma_{ij}\colonequals(\redb\Omega_{i})\cap(\redb\Omega_{j}), \Sigma_{ij}'\colonequals(\redb\Omega_{i}')\cap(\redb\Omega_{j}')$ for all $1\leq i<j\leq m$.

\begin{lemma}[\embolden{The First Variation for Minimizers}]\label{firstvarmaxnsn}
Suppose $\Omega_{1},\ldots,\Omega_{m},\Omega_{1}',\ldots,\Omega_{m}'\subset\R^{\adimn}$ minimize Problem \ref{prob2n}.  Then for all $1\leq i<j\leq m$, there exists $c_{ij},c_{ij}'\in\R$ such that
$$T_{\rho}(1_{\Omega_{i}}-1_{\Omega_{j}})(x)=c_{ij},\qquad\forall\,x\in\Sigma_{ij}.$$
$$T_{\rho}(1_{\Omega_{i}'}-1_{\Omega_{j}'})(x)=c_{ij}',\qquad\forall\,x\in\Sigma_{ij}.$$
\end{lemma}

We denote $N_{ij}(x)$ as the unit exterior normal vector to $\Sigma_{ij}$ for all $1\leq i<j\leq m$.  Also denote $N_{ij}'(x)$ as the unit exterior normal vector to $\Sigma_{ij}'$ for all $1\leq i<j\leq m$.  Let $\Omega_{1},\ldots,\Omega_{m},\Omega_{1}',\ldots,\Omega_{m}'\subset\R^{\adimn}$ be a partition of $\R^{\adimn}$ into measurable sets such that $\partial\Omega_{i},\partial\Omega_{i}'$ are a locally finite union of $C^{\infty}$ manifolds for all $1\leq i\leq m$.  Let $X,X'\in C_{0}^{\infty}(\R^{\adimn},\R^{\adimn})$.  Let $\{\Omega_{i}^{(s)}\}_{s\in(-1,1)}$ be the variation of $\Omega_{i}$ corresponding to $X$ for all $1\leq i\leq m$.  Let $\{\Omega_{i}^{'(s)}\}_{s\in(-1,1)}$ be the variation of $\Omega_{i}'$ corresponding to $X'$ for all $1\leq i\leq m$.  Denote $f_{ij}(x)\colonequals\langle X(x),N_{ij}(x)\rangle$ for all $x\in\Sigma_{ij}$ and $f_{ij}'(x)\colonequals\langle X'(x),N_{ij}'(x)\rangle$ for all $x\in\Sigma_{ij}'$.  We let $N$ denote the exterior pointing unit normal vector to $\redb\Omega_{i}$ for any $1\leq i\leq m$ and we let $N'$ denote the exterior pointing unit normal vector to $\redb\Omega_{i}'$ for any $1\leq i\leq m$.

\begin{lemma}[\embolden{Volume Preserving Second Variation of Minimizers, Multiple Sets}]\label{lemma7rn}
Let $\Omega_{1},\ldots,\Omega_{m},\Omega_{1}',\ldots,\Omega_{m}'\subset\R^{\adimn}$ be two partitions of $\R^{\adimn}$ into measurable sets such that $\partial\Omega_{i},\partial\Omega_{i}'$ are a locally finite union of $C^{\infty}$ manifolds for all $1\leq i\leq m$.  Then
\begin{equation}\label{four32pv2n}
\begin{aligned}
&\frac{\d^{2}}{\d s^{2}}\Big|_{s=0}\sum_{i=1}^{m}\int_{\R^{\adimn}} \int_{\R^{\adimn}} 1_{\Omega_{i}^{(s)}}(y)G(x,y) 1_{\Omega_{i}^{'(s)}}(x)\,\d x\d y\\
&\qquad\qquad\qquad=\sum_{1\leq i<j\leq m}\int_{\Sigma_{ij}'}\Big[\Big(\int_{\redb\Omega_{i}}-\int_{\redb\Omega_{j}}\Big)G(x,y)\langle X(y),N(y)\rangle \,\d y\Big] f_{ij}'(x) \,\d x\\
&\qquad\qquad\qquad\qquad+\sum_{1\leq i<j\leq m}\int_{\Sigma_{ij}}\Big[\Big(\int_{\redb\Omega_{i}'}-\int_{\redb\Omega_{j}'}\Big)G(x,y)\langle X'(y),N'(y)\rangle \,\d y\Big] f_{ij}(x) \,\d x\\
&\qquad\qquad\qquad\qquad\qquad+\int_{\Sigma_{ij}'}\vnormf{\overline{\nabla} T_{\rho}(1_{\Omega_{i}}-1_{\Omega_{j}})(x)}(f_{ij}'(x))^{2} \gamma_{\adimn}(x)\,\d x\\
&\qquad\qquad\qquad\qquad\qquad+\int_{\Sigma_{ij}}\vnormf{\overline{\nabla} T_{\rho}(1_{\Omega_{i}'}-1_{\Omega_{j}'})(x)}(f_{ij}(x))^{2} \gamma_{\adimn}(x)\,\d x.
\end{aligned}
\end{equation}
Also,
\begin{equation}\label{nabeq3n}
\begin{aligned}
\overline{\nabla}T_{\rho}(1_{\Omega_{i}}-1_{\Omega_{j}})(x)&=N_{ij}'(x)\vnormf{\overline{\nabla}T_{\rho}(1_{\Omega_{i}}-1_{\Omega_{j}})(x)},\qquad\forall\,x\in\Sigma_{ij}'.\\
\overline{\nabla}T_{\rho}(1_{\Omega_{i}'}-1_{\Omega_{j}'})(x)&=N_{ij}(x)\vnormf{\overline{\nabla}T_{\rho}(1_{\Omega_{i}'}-1_{\Omega_{j}'})(x)},\qquad\forall\,x\in\Sigma_{ij}.\\
\end{aligned}
\end{equation}
Moreover, $\vnormf{\overline{\nabla} T_{\rho}(1_{\Omega_{i}}-1_{\Omega_{j}})(x)}>0$ for all $x\in\Sigma_{ij}'$, except on a set of Hausdorff dimension at most $\sdimn-1$, and $\vnormf{\overline{\nabla} T_{\rho}(1_{\Omega_{i}'}-1_{\Omega_{j}'})(x)}>0$ for all $x\in\Sigma_{ij}$, except on a set of Hausdorff dimension at most $\sdimn-1$.
\end{lemma}
Equation \eqref{nabeq3n} and the last assertion require a slightly different argument than previously used.  To see the last assertion, note that if there exists $1\leq i<j\leq m$ such that $\vnorm{\overline{\nabla} T_{\rho}(1_{\Omega_{i}}-1_{\Omega_{j}})(x)}=0$ on an open set in $\Sigma_{ij}'$, then choose $X'$ supported in this open set so that the third term of \eqref{four32pv2n} is zero.  Then, choose $Y$ such that sum of the first two terms in \eqref{four32pv2n} is negative.  Multiplying then $X$ by a small positive constant, and noting that the fourth term in \eqref{four32pv2n} has quadratic dependence on $X$, we can create a negative second derivative of the noise stability, giving a contradiction.  We can similarly justify the positive signs appearing in \eqref{nabeq3n} (as opposed to the negative signs from \eqref{nabeq3}).

%

Let $v\in\R^{\adimn}$.  For simplicity of notation, we denote $\langle v,N\rangle$ as the collection of functions $(\langle v,N_{ij}\rangle)_{1\leq i<j\leq m}$ and we denote $\langle v,N'\rangle$ as the collection of functions $(\langle v,N_{ij}'\rangle)_{1\leq i<j\leq m}$.  For any $1\leq i<j\leq m$, define
\begin{equation}\label{sdef2n}
\begin{aligned}
S_{ij}(\langle v,N\rangle)(x)
&\colonequals (1-\rho^{2})^{-(\adimn)/2}(2\pi)^{-(\adimn)/2}\Big(\int_{\partial\Omega_{i}}-\int_{\partial\Omega_{j}}\Big)\langle v,N(y)\rangle e^{-\frac{\vnorm{y-\rho x}^{2}}{2(1-\rho^{2})}}\,\d y,
\,\forall\,x\in\Sigma_{ij'}\\
S_{ij}'(\langle v,N'\rangle)(x)
&\colonequals (1-\rho^{2})^{-(\adimn)/2}(2\pi)^{-(\adimn)/2}\Big(\int_{\partial\Omega_{i}'}-\int_{\partial\Omega_{j}'}\Big)\langle v,N'(y)\rangle e^{-\frac{\vnorm{y-\rho x}^{2}}{2(1-\rho^{2})}}\,\d y
\,\forall\,x\in\Sigma_{ij}.
\end{aligned}
\end{equation}

\begin{lemma}[\embolden{Key Lemma, $m\geq 2$, Translations as Almost Eigenfunctions}]\label{treig2n}
Let $\Omega_{1},\ldots,\Omega_{m},\Omega_{1}',\ldots,\Omega_{m}'\subset\R^{\adimn}$ minimize problem \ref{prob2n}.  Fix $1\leq i<j\leq m$.  Let $v\in\R^{\adimn}$.  Then
\begin{flalign*}
S_{ij}(\langle v,N\rangle)(x)
&=-\langle v,N_{ij}'(x)\rangle\frac{1}{\rho}\vnormf{\overline{\nabla} T_{\rho}(1_{\Omega_{i}}-1_{\Omega_{j}})(x)},\qquad\forall\,x\in\Sigma_{ij}'.\\
S_{ij}'(\langle v,N'\rangle)(x)
&=-\langle v,N_{ij}(x)\rangle\frac{1}{\rho}\vnormf{\overline{\nabla} T_{\rho}(1_{\Omega_{i}'}-1_{\Omega_{j}'})(x)},\qquad\forall\,x\in\Sigma_{ij}.\\
\end{flalign*}
\end{lemma}
When compared to Lemma \ref{treig2}, Lemma \ref{treig2n} has a negative sign on the right side of the equality, resulting from the positive sign in \eqref{nabeq3n} (as opposed to the negative sign on the right side of \eqref{nabeq3}).  Lemmas \ref{lemma7rn} and \ref{treig2n} then imply the following.

\begin{lemma}[\embolden{Second Variation of Translations, Multiple Sets}]\label{keylemn}
Let $0<\rho<1$.  Let $v\in\R^{\adimn}$.  Let $\Omega_{1},\ldots,\Omega_{m}$ minimize problem \ref{prob2}.  For each $1\leq i\leq m$, let $\{\Omega_{i}^{(s)}\}_{s\in(-1,1)}$ be the variation of $\Omega_{i}$ corresponding to the constant vector field $X\colonequals v$.  Assume that
$$\int_{\partial\Omega_{i}}\langle v,N(x)\rangle \gamma_{\adimn}(x)\,\d x=\int_{\partial\Omega_{i}'}\langle v,N(x)\rangle \gamma_{\adimn}(x)\,\d x=0,\qquad\forall\,1\leq i\leq m.$$
Then
\begin{flalign*}
&\frac{\d^{2}}{\d s^{2}}\Big|_{s=0}\sum_{i=1}^{m}\int_{\R^{\adimn}}1_{\Omega_{i}^{(s)}}(x)T_{\rho}1_{\Omega_{i}^{'(s)}}(x)\gamma_{\adimn}(x)\,\d x\\
&\qquad\qquad\qquad=\Big(-\frac{1}{\rho}+1\Big)\sum_{1\leq i<j\leq m}\int_{\Sigma_{ij}}\vnormf{\overline{\nabla}T_{\rho}(1_{\Omega_{i}'}-1_{\Omega_{j}'})(x)}\langle v,N_{ij}(x)\rangle^{2}\gamma_{\adimn}(x)\,\d x\\
&\qquad\qquad\qquad\,\,+\Big(-\frac{1}{\rho}+1\Big)\sum_{1\leq i<j\leq m}\int_{\Sigma_{ij}'}\vnormf{\overline{\nabla}T_{\rho}(1_{\Omega_{i}}-1_{\Omega_{j}})(x)}\langle v,N_{ij}'(x)\rangle^{2}\gamma_{\adimn}(x)\,\d x.
\end{flalign*}
\end{lemma}

Since $\rho\in(0,1)$, $-\frac{1}{\rho}+1<0$.  (The analogous inequality in Lemma \ref{keylem} was $\frac{1}{\rho}-1>0$.)  Repeating the argument of Theorem \ref{mainthm1} then gives the following.

\begin{theorem}[\embolden{Main Structure Theorem/ Dimension Reduction, Negative Correlation}]\label{mainthm1n}
Fix $0<\rho<1$.  Let $m\geq2$ with $2m\leq \sdimn+3$.  Let $\Omega_{1},\ldots\Omega_{m},\Omega_{1}',\ldots\Omega_{m}'\subset\R^{\adimn}$ minimize Problem \ref{prob2n}.  Then, after rotating the sets $\Omega_{1},\ldots\Omega_{m},\Omega_{1}',\ldots\Omega_{m}'$ and applying Lebesgue measure zero changes to these sets, there exist measurable sets $\Theta_{1},\ldots\Theta_{m},\Theta_{1}',\ldots\Theta_{m}'\subset\R^{2m-2}$ such that,
$$\Omega_{i}=\Theta_{i}\times\R^{\sdimn-2m+3},\,\,\Omega_{i}'=\Theta_{i}'\times\R^{\sdimn-2m+3}\qquad\forall\, 1\leq i\leq m.$$  
\end{theorem}

\medskip

\bibliographystyle{amsalpha}
\bibliography{12162011}

\end{document}